\documentclass[a4paper,11pt]{article}
\usepackage{amssymb,amsthm,amsfonts,amstext,arev,bbm}
\usepackage{amsmath}
\usepackage{mathptmx}       
\usepackage{helvet}         
\usepackage{courier}        

\def\clap#1{\hbox to 0pt{\hss#1\hss}}
\def\mathllap{\mathpalette\mathllapinternal}
\def\mathrlap{\mathpalette\mathrlapinternal}
\def\mathclap{\mathpalette\mathclapinternal}
\def\mathllapinternal#1#2{%
           \llap{$\mathsurround=0pt#1{#2}$}}
\def\mathrlapinternal#1#2{%
           \rlap{$\mathsurround=0pt#1{#2}$}}
\def\mathclapinternal#1#2{%
           \clap{$\mathsurround=0pt#1{#2}$}}

\usepackage{color}

\newtheorem{theorem}{Theorem}
\newtheorem{lemma}[theorem]{Lemma}
\newtheorem{proposition}[theorem]{Proposition}

\theoremstyle{remark}

\newcommand{\mc}[1]{{\mathcal #1}}

\newcommand{\mb}[1]{{\mathbf #1}}
\newcommand{\bb}[1]{{\mathbb #1}}

\newcommand{\<}{\langle}
\renewcommand{\>}{\rangle}
\newcommand{\linn}{\langle\!\!\langle}
\newcommand{\rinn}{\rangle\!\!\rangle}

\DeclareMathOperator{\supp}{supp}

\DeclareMathOperator*{\nop}{\smash{\mathrm{sup}}}

\newcommand{\Iex}{\mc I_{\mathrm{ex}}}
\newcommand{\Irw}{\mc I_{\mathrm{rw}}}

\begin{document}

\title{Explicit LDP for a slowed RW driven by a symmetric exclusion process}

\author{\renewcommand{\thefootnote}{\arabic{footnote}}
L.\ Avena \footnotemark[1],
\renewcommand{\thefootnote}{\arabic{footnote}}
M.\ Jara \footnotemark[2],
\renewcommand{\thefootnote}{\arabic{footnote}}
F.\ V\"ollering \footnotemark[3]}

\footnotetext[1]{MI, University of Leiden, The Netherlands. E-mail: l.avena@math.leidenuniv.nl}
\footnotetext[2]{IMPA, Rio de Janeiro, Brazil. E-mail: mjara@impa.br}
\footnotetext[3]{University of M\"unster, Germany. E-mail: f.voellering@uni-muenster.de}

\maketitle

\begin{abstract}
We consider a random walk (RW) driven by a simple symmetric exclusion process (SSE). Rescaling the RW and the SSE in such a way that a joint hydrodynamic limit theorem holds we prove a joint path large deviation principle. The corresponding large deviation rate function can be split into two components, the rate function of the SSE and the one of the RW given the path of the SSE. 
These components have different structures (Gaussian and Poissonian, respectively) and to overcome this difficulty we make use of the theory of Orlicz spaces.
In particular, the component of the rate function corresponding to the RW is explicit.

\vspace{0.5cm}\noindent
\\
{\it Acknowledgments:} L.A. has been supported by NWO Gravitation Grant 024.002.003-NETWORKS.
\end{abstract}

\newpage


\section{Introduction}

\subsection{Background}
Random evolution on a random medium has been the object of intensive research within the mathematics and physics communities over at least the last forty years. 
Although there are plenty of rigorous and non-rigorous results obtained through a wide range of techniques and methods, it is far from being a closed subject.
Since the works of Solomon \cite{Sol}, Harris \cite{Har} and Spitzer \cite{Spi}, random walks in both static and dynamic random environments have been a prolific way to study this problem within the context of probability theory (see \cite{Z} for a review in the static case). In the case of {\em dynamic} random environments, considerable progress has been recently achieved (see e.g. \cite{AvedSanVol}, \cite{PerStaSte}, \cite{RedVol} and references therein), but a key ingredient common to all these developments is the availability of {\em good mixing properties} of the environment. More recently, examples of dynamic random environments with less restrictive mixing properties have been considered \cite{HilTeidSanSiddHol}, but the general picture is still far from being understood (see \cite{AveTho} for some conjectures based on simulations). 

A very simple way to obtain a family of dynamic random environments with poor mixing properties is to consider  {\em conservative particle systems} as dynamic environments. On the one hand these environment processes are very well understood, in particular their mixing properties are well known, on the other hand the essential difficulties coming from the poor mixing are encoded into the conservation laws. 

In this article the dynamic environment is given by a simple, symmetric exclusion process, as in \cite{AveFraJarVol}, \cite{ARdH}, \cite{AvedSanVol}, \cite{AveTho}, \cite{HuvSim}, \cite{KV}, \cite{SV}. On top of this environment, we run a simple random walk with jumping rates depending on the portion of the environment it sees. The exclusion particles do not feel the presence of the random walk. 
We introduce a scaling parameter $n \in \bb N$ and we {\em speed up} the exclusion particles with respect to the random walk by a factor of $n$. Although this speeding up seems to be there in order to give us the necessary mixing properties for the environment, this is not the case. At least at a formal level what happens is that this scale is the {\em crossover scale} between a regime on which the environment behaves essentially as frozen from the point of view of the random walker, and a regime on which the environment mixes fast enough to put us back on the setting of previous works.

\subsection{The model and result at a glance: an example}

Let us consider the following problem. For the sake of clarity, we begin by considering a simple case. In Section \ref{s2} we define the general model. Let $n \in \bb N$ be a scaling parameter, which will be sent to $+\infty$ later on. On a discrete circle $\bb T_n$ with $n$ points, we run a symmetric, simple exclusion process $\{\eta_t^n; t \in [0,T]\}$,\footnote{The time window $[0,T]$ is chosen to be of finite size to avoid some technical topological considerations. Indeed, $[0,\infty)$ is not compact and one would have to be more careful in weighting the tails near infinity.} speeded up by $n^2$. We call this process the {\em dynamic environment}. Given a realization of the process $\{\eta_t^n;  t \in [0,T]\}$, we run a simple random walk on $\bb T_n$ with the following dynamics. The walk waits an exponential time of rate $n$, at the end of which it jumps to the left with probability $\tfrac{1}{3}$, it jumps to the right with probability $\tfrac{1}{3}$ and with probability $\tfrac{1}{3}$ it looks at the {\em environment} $\eta_t^n$. Let $x$ be the current position of the walk. If $\eta_t^n(x) =1$, the walk jumps to the right, and if $\eta_t^n(x) =0$ the walk jumps to the left. Notice that the particle is {\em speeded up by }$n$. Let us think about the circle $\bb T_n$ as a discrete approximation of the continuous circle of length 1. The different speeds of the environment and the walk are taken in such a way that the environment has a {\em diffusive scaling} and the walk has a {\em ballistic} (or {\em hyperbolic}  in the terminology of hydrodynamic limits) scaling. Let us start the exclusion process from a non-equilibrium initial distribution. In order to fix ideas, imagine that $\eta_0^n(x) = 1$ if $|x| \leq \frac{n}{4}$ and $\eta_0^n(x)=0$ otherwise. This initial distribution of particles is a discrete approximation of the density profile $u_0(x) = \mathbf{1}_{|x| \leq 1/4}$. It is precisely under this diffusive space-time scaling that the limiting density profile has a non-trivial evolution. This limiting profile $u(t,x)$ turns out to be the solution of the heat equation on the continuous circle, with initial condition $u_0$. This convergence is what is known in the literature as the {\em hydrodynamic limit} of the exclusion process (see Chapter 4 of \cite{KL} for more details and further references). Now let us describe the scaling limit of the walk. If the density of particles of the exclusion process is equal to $\rho \in [0,1]$, then one expects that the walk will move with velocity $v(\rho) = \frac{1}{3} (2\rho-1)$. Notice that a hyperbolic scaling is needed for the walk in order to have a non-trivial macroscopic velocity. Therefore, the macroscopic position of the walk should satisfy the ODE 
\begin{equation*}
{\dot \varphi}_t = v(u(t,\varphi_t)).
\end{equation*}
This heuristic reasoning has been made precise in \cite{AveFraJarVol} in the form of a functional weak law of large numbers for the walk. We obtain in this paper a large deviation principle associated to this law of large numbers. The form of the rate function associated to this large deviation principle is given by the variational formula
\begin{equation*}
\mc I(x) = \inf_{\pi} \big\{ \mc I_{\mathrm{rw}}(x|\pi) + \mc I_{\mathrm{ex}}(\pi)\big\},
\end{equation*}
where $\mc I_{\mathrm{rw}}(x|\pi)$ is the rate function of a random walk on a given space-time realization of the environment $\pi$ and $\mc I_{\mathrm{ex}}(\pi)$ is the rate function of the large deviation principle associated to the hydrodynamic limit of the exclusion process (see Section \ref{s3} for more precise definitions). This variational formula is very reminiscent of the variational formula relating the quenched and averaged large deviation principles for random walks in random environments \cite{CGZ},\cite{GdH}, see in particular Eq.~(9) of \cite{CGZ}. Notice that in our setting a ``quenched" large deviation principle or even a quenched law of large numbers is out of reach since the exclusion process does not have an almost sure hydrodynamic limit. Anyway, the interpretation of the variational formula is the same as the corresponding one for random walks in random environments. The function $\mc I_{\mathrm{rw}}(x|\pi)$ is the cost of observing a trajectory $x$ when the environment has a space-time density $\pi$, and $\mc I_{\mathrm{ex}}(\pi)$ is the cost of changing the density of the environment to $\pi$.

Our method of proof, however, differs from the one in \cite{CGZ}. In \cite{AveFraJarVol}, we proved a joint law of large numbers for the environment and the walk. We show in this article that the rate function of the corresponding large deviation principle is given by $\mc I_{\mathrm{rw}}(x|\pi) + \mc I_{\mathrm{ex}}(\pi)$. The desired result follows as an application of the contraction principle. 

\subsection{Discussion}\label{discussion}

There are not many works addressing the question of large deviations for random walks in dynamic random environment. In \cite{ARdH}, the authors show a large deviation principle (LDP) for the empirical speed on some attractive random environments. They also show that the rate function in the case of an exclusion process as a random environment has a flat piece. This reference is the closest in spirit to our work. To our knowledge, the earliest reference in this field seems to be  \cite{I-R}, then, in a series of papers, \cite{CDRR-AS}, \cite{R-ASY}, \cite{RedVol}, \cite{Yil} the authors show an LDP for fairly general dynamic random environments. In \cite{dhS}, the authors gave an LDP for a random walk driven by a contact process. In all of these results, the environment is Markovian (except fo the more general setting in \cite{R-ASY}) and it is assumed to start from an {\em ergodic equilibrium}. One of the differences of our work with respect to these results is that we consider environments which start from a {\em local equilibrium}, see \eqref{ec18}. These environments are more general than ergodic equilibria and they give rise to a richer phenomenology. Our variational formula for the rate function could in principle be explicit enough to allow some finer analysis of the behavior of the walk, but we do not pursue this line of research here.

Our method of proof is very different from what has been done before, and as mentioned above it relies on a joint LDP for the couple environment-random walk. The large deviations of the environment are quadratic in nature, since large fluctuations are built up on small, synchronised variations  of the behavior of individual particles, and the large deviations of the random walk are exponential in nature due to the Poissonian structure of the walk. For this reason the joint LDP proved to be very difficult to obtain. In particular, we need to deal with non-convex entropy cost functions. 

From the point of view of interacting particle systems, the problem addressed in this work is close in spirit to the problem of the behavior of a tagged particle in the exclusion process. In fact, we borrowed from \cite{JLS} the strategy of proof of the joint environment-walk law of large numbers, although this strategy can be traced back to the seminal article \cite{KV}. Recently, an LDP for the tagged particle in one-dimensional, nearest-neighbor symmetric exclusion process has been obtained \cite{SV}. On the one hand, the results in \cite{SV} are more demanding, because the motion of the tagged particle affects the motion of the environment in a sensitive way.  On the other hand, our result is more intricate because of the mixture between Poissonian and Gaussian rate functions. This last point obliges us to use the machinery of {\em Orlicz spaces} in order to show that the variational problem that defines the rate function is well-posed. In the realm of interacting particle systems, this kind of problems poses real difficulties in order to obtain an LDP. A family of models which shares the difficulties found in this work is a conservative dynamics superposed to a creation-annihilation mechanism. To our knowledge, the best result so far is found in \cite{BL}. In that article, a creation-annihilation (or {\em Glauber}) mechanism is superposed to the exclusion dynamics with a speeding up of the exclusion process in order to make both dynamics relevant in the macroscopic limit. As in our case, the rate function of the LDP can be written as a combination of the Gaussian rate function of the exclusion process and a Poissonian rate function coming from the Glauber dynamics. However, they impose an additional condition (see Assumption (L1) on page 8 of \cite{BL}) which makes some key cost functions convex. This point is very technical but also very delicate, and it is the key to proving that the upper and lower bounds match. We overcame this problem by using the theory of Orlicz spaces, see Section \ref{s7.3}.

\subsection{Organization of the article}

In Section \ref{s2} we describe our model in full generality. We fix some notation and in particular we introduce the {\em environment as seen by the walker}, which will be very important in order to relate the behaviors of the walk and of the environment. We also describe the hydrodynamic limits associated to the exclusion process, as well as to the environment as seen by the walker. This part summarizes the functional law of large numbers obtained in \cite{AveFraJarVol}.  In Section \ref{s3} we start explaining what we understand by a large deviation principle for the couple environment-walk. We put some emphasis on the topologies considered for the process, since they are not the standard ones. In particular, we look at the random walk as a {\em signed} Poisson point process. The trajectory of the random walk can be easily recovered from this process and vice-versa, but the topology of signed measures turns out to be more convenient. We finally state our main result, Theorem \ref{t2} on page 11 which is a large deviation principle for the couple environment-walk. The large deviation principle for the walk, Theorem \ref{t1},  follows at once from Theorem \ref{t2} via the contraction principle. 
In Section \ref{Martingales} we define some exponential martingales which will be used to tilt our dynamics, following the usual Donsker-Varadhan (see e.g. \cite{DV}) strategy of proof for large deviations of Markov processes.
In Section \ref{s4} we show what is called in the literature the {\em superexponential lemma}. This lemma allows to do two things. First, it allows to write the exponential martingales introduced in Section \ref{Martingales} as functions of the couple environment-walk plus an error term which is superexponentially small. This step is the starting point of the upper bound. And second, it allows to obtain the hydrodynamic limit of suitable perturbations of the dynamics. The latter is the starting point of the lower bound.
In Section \ref{s5} we show an energy estimate. This energy estimate allows to restrict our considerations to the space of measures with {\em finite energy} with respect to the Lebesgue measure. In particular, all these measures will be absolutely continuous with respect to Lebesgue measure. This point is crucial, since we need to evaluate this density at the location of the random walk in order to know its local drift. 
In Section \ref{s6} we prove the large deviation upper bound and in Section \ref{s7} we prove a matching lower bound, which finishes the proof of the large deviation principle for the couple environment-walk.

\section{The model}
\label{s2} 
\subsection{The environment}
\label{s1.1}
Let $n \in \bb N$ be a scaling parameter, $\bb T_n = \frac{1}{n} \bb Z / \bb Z$ be the discrete circle of size $n$ and $\Omega_n = \{0,1\}^{\bb T_n}$. We denote by $\eta = \{\eta(x); x \in \bb T_n\}$ the elements of $\Omega_n$ and we call $\eta$ a {\em configuration of particles}. The elements $x$ of $\bb T_n$ will be called {\em sites}, and we say that there is a particle at site $x \in \bb T_n$ in configuration $\eta$ if $\eta(x) =1$. Otherwise, we declare the site $x$ to be {\em empty}. We say that $x,y \in \bb T_n$ are {\em neighbours} if $|y-x|=\frac{1}{n}$. In this case we write $x \sim y$. Fix $T>0$. The simple, symmetric {\em exclusion process} on $\bb T_n$ is the continuous-time Markov process $\{\eta_t^n; t \in [0,T]\}$ with the following dynamics. To each pair of neighbours $\{x,y\}$ on $\bb T_n$ we attach a Poisson clock of rate $n^2$, independent of the other clocks. Each time the clock associated to the pair $\{x,y\}$ rings, we exchange the values of $\eta_t^n(x)$ and $\eta_t^n(y)$.

For $\eta \in \Omega_n$ and $x,y \in \bb T_n$, we define $\eta^{x,y} \in \Omega_n$ as
\begin{equation*}
\eta^{x,y}(z) =
\begin{cases}
\eta(y); &z=x,\\
\eta(x); &z=y,\\
\eta(z); &z \neq x,y.
\end{cases}
\end{equation*}
The process $\{\eta_t^n; t \in [0,T]\}$ is generated by the operator given by
\begin{equation}\label{exGen}
L^{ex}_n f(\eta) = n^2 \sum_{x \sim y}\big(f(\eta^{x,y})-f(\eta)\big)
\end{equation}
for any $f: \Omega_n \to \bb R$. Notice that if the initial configuration $\eta_0^n$ has only one particle, this particle follows a simple random walk. This fact explains the acceleration $n^2$ in the dynamics, corresponding to a diffusive space-time scaling. We consider the process defined on a finite time window $[0,T]$ to avoid uninteresting topological issues (see the footnote on page 2).  

By reversibility and irreducibility, for each $k \in \{0,1,\dots,n\}$, the uniform measure $\nu_{k,n}$ on 
\begin{equation*}
\Omega_{n,k} = \Big\{\eta \in \Omega_n; \sum_{x \in \bb T_n} \eta(x) =k\Big\}
\end{equation*}
is invariant and ergodic under the evolution of $\{\eta_t^n; t \in [0,T]\}$. Equivalently, for each $\rho \in [0,1]$ the product Bernoulli measure $\nu_\rho$ on $\Omega_n$, defined by
\begin{equation*}
\nu_\rho(\eta) = \prod_{x \in \bb T_n} \big\{ \rho \eta(x) + (1-\rho)(1-\eta(x))\big\}
\end{equation*}
is invariant under the evolution of $\{\eta_t^n; t \in [0,T]\}$.

\subsection{Some notation}

For $x \in \bb T_n$, let $\tau_x: \Omega_n \to \Omega_n$ be the canonical shift, that is, $\tau_x \eta(z) = \eta(z+x)$ for any $\eta \in \Omega_n$ and any $z \in \bb T_n$. For $f : \Omega_n \to \bb R$ we define $\tau_x f: \Omega_n \to \bb R$ as $\tau_x f(\eta) = f(\tau_x \eta)$ for any $\eta \in \Omega_n$. 

We say that a set $A \subseteq \bb T_n$ is the {\em support} of a function $f: \Omega_n \to \bb R$ if:
\begin{itemize}
\item[i)] for any $\eta, \xi \in \Omega_n$ such that $\eta(x)= \xi(x)$ for all $x \in A$, $f(\eta) = f(\xi)$,
\item[ii)] $A$ is the smallest set satisfying i).
\end{itemize}
We denote this by $A = \supp(f)$.

Let $\Pi: \bb Z \to \bb T_n$ be the unique map from $\bb Z$ to $\bb T_n$ such that $\Pi(0)=0$ and $\Pi(x+1)-\Pi(x) = \frac{1}{n}$ for any $x \in \bb Z$, that is, $\Pi$ is the {\em canonical covering} of $\bb T_n$ by $\bb Z$.
Consider $\Omega = \{0,1\}^{\bb Z}$. We say that a function $f: \Omega \to \bb R$ is {\em local} if there exists a finite $A \subseteq \bb Z$ such that for any $\eta, \xi \in \Omega$ with $\eta(x)= \xi(x)$ for all $x \in A$, $f(\eta) = f(\xi)$. For a local function $f: \Omega \to \bb R$, we can define $\supp(f)$ as above. We can identify $\Omega_n$ with the set $\{0,1\}^{\{\lfloor-\frac{n}{2}+1\rfloor,\dots,\lfloor \frac{n}{2}\rfloor\}}$. 
Using this identification, any local function $f: \Omega \to \bb R$ can be lifted to a function (which we still denote by $f$) from $\Omega_n$ to $\bb R$, for any $n$ large enough. Moreover, under this convention, the lifting is unique. We will use the following notation. A local function $f: \Omega_n \to \bb R$ is actually a family of functions $\{f_n:\Omega_n \to \bb R; n \geq n_0\}$, all of them lifted to $\Omega_n$ from a common function $f: \Omega \to \bb R$, which we assume to be local. For a local function $f: \Omega_n \to \bb R$, $\supp(f)$ will denote either the support of $f $ on $\bb Z$ or the support of $f_n$ on $\bb T_n$, which is equal to $\Pi(\supp(f))$.

\subsection{The random walk}
Let $c: \Omega\times\{+,-\} \to [0,\infty)$ be a local function, and let $c_n$ be the lifted version on $\Omega_n$. Define $c^\pm_n:\Omega_n \times \bb T_n$ via the {\em cocycle property}: $c^{\pm}_n(\eta;x) =  c_n(\tau_x\eta,\pm)$ for any $\eta \in \Omega_n$ and any $x \in \bb T_n$. As the dependence on $n$ is clear from context we simply write $c^\pm$. Without loss of generality we only consider $n$ large enough so that the lifting of $c$ exists.
We call $c$ a {\em jump rate}. An archetypical example is
\begin{equation*}
c^+(\eta; x) = \alpha + (\beta-\alpha)\eta(x), \quad c^-(\eta;x) = \beta +(\alpha-\beta)\eta(x), \quad \text{for some } \alpha,\beta>0.
\end{equation*} 

The random walk in {\em dynamic random environment} $\{\eta_t^n; t \in [0,T]\}$ with jump rate $c$ is the continuous-time Markov process $\{x_t^n; t \in [0,T]\}$ with values in $\bb T_n$ with the following dynamics. For simplicity, assume that $c^++c^- \equiv 1$, the reader can see that this assumption is not relevant. We attach to a random walker a Poisson clock of rate $n$, independent of the process $\{\eta_t^n; t \in [0,T]\}$. Each time the clock rings, the particle jumps to the right with probability $c^+(\eta_t^n; x_{t-}^n)$, and to the left with complementary probability $c^-(\eta_t^n; x_{t-}^n)$. We remark that the process $\{x_t^n; t \in [0,T]\}$ itself is {\em not} Markovian; if we consider a {\em fixed} realization of the random environment $\{\eta_t^n; t \in [0,T]\}$, then we recover the Markov property for $\{x_t^n;t \in [0,T]\}$, but the resulting evolution is not homogeneous in time. The pair $\{(\eta_t^n; x_t^n); t \in [0,T]\}$ turns out to be an homogeneous Markov process, with values in $\Omega_n \times \bb T_n$ and generated by the operator given by
\begin{equation}\label{jointGen}
L_n f(\eta; x) = n^2 \sum_{y\sim z} \big(f(\eta^{y,z};x)-f(\eta;x)\big) +n\sum_{z=\pm 1} c^z(\eta;x)\big(f(\eta;x+\tfrac{z}{n})-f(\eta;x)\big)
\end{equation}
for any function $f: \Omega_n \times \bb T_n \to \bb R$. At this point, two remarks are in place. Notice that for functions which depend only on $\eta$, this expression coincides with the definition of the generator of the process $\{\eta_t^n;t \in [0,T]\}$, explaining the use of the same notation for both objects. Notice as well that the dynamics of the random walk is speeded-up by $n$. We expect the walk to move with some velocity, in which case it needs to make $n$ jumps in order to cross a region of order 1.

From now on and up to the end of the article, we assume that the random walk starts at $0$: $x_0^n=0$ for any $n \in \bb N$. 

\subsection{The environment as seen by the walker} Let $\{\xi_t^n; t \in [0,T]\}$ be the process with values in $\Omega_n$ defined by $\xi_t^n(z) = \eta_t^n(x_t^n+z)$ for any $z\in \bb T_n$ (in other words, $\xi_t^n = \tau_{x_t^n} \eta_t^n$) and any $t \in [0,T]$. The process $\{\xi_t^n; t \in [0,T]\}$ turns out to be a Markov process and its corresponding generator is given by 
\begin{equation}\label{EPGen}
\mc L_n f(\xi) =  n^2 \sum_{x\sim y} \big(f(\xi^{x,y})-f(\xi)\big) +n\sum_{z=\pm 1} c^z(\xi;0)\big(f(\tau_{\frac{z}{n}} \xi)-f(\xi)\big)
\end{equation}
for any function $f:\Omega_n \to \bb R$. The value of $x_t^n$ can be recovered from the {\em trajectory} $\{\xi_s^n; s \in [0,t]\}$ in the following way. First suppose that $\xi^n$ has at least 2 particles and two empty sites. Let $\{N_t^{n,\pm}; t \in [0,T]\}$ be the number of shifts to the right ($+$) and to the left ($-$) up to time $t$. Then,
\begin{equation*}
x_t^n = \Pi\big(N_t^{n,+}-N_t^{n,-}\big).
\end{equation*}
If there is only one particle or one empty site, $N_t^{n,\pm}$ are similar, but each right (left) shift of $\xi^n_t$ is discarded with probability $n/(n+c^-(\xi^n_t;0))$ ($n/(n+c^+(\xi^n_t;0))$), which is the probability that the observed shift came from the movement of the single particle/empty site. If there are no particles/empty sites, $N_t^{n,\pm}$ are Poisson processes with rate $n c^\pm(\mb 0;0)$ or $n c^\pm(\mb 1;0)$, where $\mb 0$ and $\mb 1$ are the empty and full configurations.

This point of view, the environment as seen by the walker, introduced by Kipnis-Varadhan \cite{KV}, has shown to be very fruitful (see \cite{AveFraJarVol} for an application in this context).

\subsection{The empirical measures}
Let $\bb T = \bb R / \bb Z$ and $\mc M^+(\bb T)$ be the space of positive Radon measures on $\bb T$. For $\mu$ and $\{\mu^n; n \in \bb N\}$ in $\mc M^+(\bb T)$, we say that $\mu^n \to \mu$ if $\int fd\mu^n \to \int f d\mu$ for any continuous function $f: \bb T \to \bb R$. The topology induced on $\mc M^+(\bb T)$ by this convergence is known as the {\em weak topology}, and $\mc M^+(\bb T)$ turns out to be a Polish space under this topology. That is, $\mc M^+(\bb T)$ is completely metrizable and separable under this topology. A possible metric is the following. Let $\{f_N; N \in \bb Z\}$ be a dense subset in $\mc C(\bb T)$. Then, $d: \mc M^+(\bb T) \times \mc M^+(\bb T) \to [0,\infty)$ given by
\begin{equation*}
d(\mu,\nu) = \sum_{N \in \bb Z} \frac{1}{2^{|N|}} \min\Big\{\Big|\int f_Nd(\mu-\nu)\Big|,1\Big\}
\end{equation*}
is the required metric.

For $x \in \bb T_n$, let $\delta_x^n: \bb T \to \bb R$ be defined as
\begin{equation*}
\delta_x^n(y)=\big(1-n|y-x|\big)^+,
\end{equation*}
where $(\cdot)^+$ denotes positive part. Sometimes the functions $\{\delta_x^n; x \in \bb T_n\}$ are called {\em finite elements}. The empirical density of particles is defined as the $\mc M^+(\bb T)$-valued process $\{\pi_t^n; t \in [0,T]\}$ given by
\begin{equation*}
\pi_t^n(dy) = \sum_{x \in \bb T_n} \eta_t^n(x) \delta_x^n(y) dy.
\end{equation*}
Notice that $\pi_t^n$ is absolutely continuous with respect to Lebesgue measure on $\bb T$. We will make the following abuse of notation. We will use $\pi_t^n$ to designate indistinctly the measure $\pi_t^n(dx)$ or its density function $\pi_t^n(\cdot)$ with respect to Lebesgue measure.
We denote by $\pi_t^n(H)$ the integral of a function $H$ with respect to the measure $\pi_t^n(dx)$. At this point, some comments about this definition are in place. It is customary in the literature of interacting particle systems to use $\frac{1}{n}\delta_x$ in place of $\delta_x^n$, where $\delta_x$ is the $\delta$ of Dirac at $x \in \bb T$ (see Chapter 4 of \cite{KL}). We will be interested in scaling limits of the process $\{\pi_t^n; t \in [0,T]\}$. Since the number of particles per site is bounded by $1$ by definition, any limit point of $\pi_t^n(dx)$ must be a measure which is absolutely continuous with respect to Lebesgue measure on $\bb T$, and moreover with Radon-Nikodym 
derivative bounded above by $1$. Therefore, it is natural to modify the customary definition of the empirical measure $\pi_t^n$ in such a way that it satisfies this property for any fixed $n$. This is accomplished by choosing $\delta_x^n(y) = \mathbf{1}(|y-x| \leq \frac{1}{2n})$ (see, e.g., \cite{KOV}). In our case, for topological considerations which will become more transparent later on, it will be convenient to have $\pi_t^n(\cdot)$ a.s.~{\em continuous}, since on one hand we will need this property later on, and on the other hand we will prove that this property is shared by the possible limits of $\pi_t^n$. It is clear that at the level of a law of large numbers, all these definitions of empirical measures are equivalent; this is also the case at the level of large deviation principles, and we adopt this definition in order to simplify the already very technical exposition.

Let us denote by $\mc M_{0,1}^+(\bb T)$ the subset of $\mc M^+(\bb T)$ formed by measures $\mu$ absolutely continuous with respect to Lebesgue measure on $\bb T$, such that $0 \leq \frac{d\mu}{dx} \leq 1$. On $\mc M_{0,1}^+(\bb T)$ we consider the weak topology defined above. Notice that $\mc M_{0,1}^+(\bb T)$ is a compact subset of $\mc M^+(\bb T)$, and $\{\pi_t^n; t \in [0,T]\}$ as defined above is an $\mc M_{0,1}^+(\bb T)$-valued process.

In a similar way, the empirical measure associated to the process $\{\xi_t^n; t \in [0,T]\}$ is defined as the $\mc M_{0,1}^+(\bb T)$-valued process $\{\hat{\pi}_t^n; t \in [0,T]\}$ given by
\begin{equation*}
\hat{\pi}_t^n(dy) = \sum_{x \in \bb T_n} \xi_t^n(x) \delta_x^n(y)dy.
\end{equation*}

\subsection{Hydrodynamic limits}
\label{s1.6}
Let $u_0: \bb T \to [0,1]$ be a given function. We say that a sequence $\{\mu^n; n \in \bb N\}$ of probability measures on $\Omega_n$ is {\em associated} to $u_0$ if for any $f \in \mc C(\bb T)$,
\begin{equation*}
\lim_{n \to \infty} \int \sum_{x \in \bb T_n} \eta(x) \delta_x^n(y) f(y) dy = \int u_0(y) f(y) dy,
\end{equation*}
in distribution with respect to $\{\mu^n; n \in \bb N\}$. In other words, $\{\mu^n; n \in \bb N\}$ is associated to $u_0$ if the empirical measure of particles converges to $u_0(y) dy$, in distribution with respect to $\{\mu^n; n \in \bb N\}$ and in the weak topology on $\mc M^+(\bb T)$. Notice that for any function $u_0: \bb T \to [0,1]$ there is a sequence of measures associated to it. Indeed, define for $n \in \bb N$ and $x \in \bb T_n$,
\begin{equation*}
\rho_x^n = n \!\!\!\!\! \int\limits_{|y-x|\leq \frac{1}{2n}} \!\!\!\!\!u_0(y) dy.
\end{equation*}
Then the product measure $\nu^n_{u_0}$ given by
\begin{equation}
\label{ec18}
\nu^n_{u_0}(\eta) = \prod_{x \in \bb T_n} \big\{ \rho_x^n \eta(x) + (1-\rho_x^n)(1-\eta(x))\big\}
\end{equation}
is associated to $u_0$. These measures will play a role in the derivation of a large deviation principle later on.

For a given Polish space $\mc E$, let $\mc D([0,T]; \mc E)$ denote the space of c\`adl\`ag trajectories from $[0,T]$ to $\mc E$. We consider on  $\mc D([0,T]; \mc E)$ the $J_1$-Skorohod topology. Let $\{\mu^n; n \in \bb N\}$ be fixed. We denote by $\bb P_n$ the distribution of $\{(\eta_t^n; x_t^n); t \in [0,T]\}$ in $\mc D([0,T]; \Omega_n \times \bb T_n)$ with initial distribution $\mu^n \otimes \delta_0$, and we denote by $\bb E_n$ the expectation with respect to $\bb P_n$. The following proposition is classical:

\begin{proposition}
\label{p1}
Fix $u_0: \bb T \to [0,1]$ and let $\{\mu^n; n \in \bb N\}$ be associated to $u_0$. With respect to $\bb P_n$,
\begin{equation*}
\lim_{n \to \infty} \pi_t^n(dx) = u(t,x) dx
\end{equation*}
in distribution with respect to the $J_1$-Skorohod topology on  $\mc D([0,T]; \mc M^+(\bb T))$, where the density $\{u(t,x); t \in [0,T], x \in \bb T\}$ is the solution of the {\em heat equation}
\begin{equation*}
\begin{cases}
\partial_t u(t,x) &= \Delta u(t,x)\\
u(0,\cdot) &=u_0(\cdot).
\end{cases}
\end{equation*}
\end{proposition}
This proposition is what is known in the literature as the {\em hydrodynamic limit} of the process $\{\eta_t^n; t \in [0,T]\}$. A proof of this proposition which is close in spirit to the exposition here can be found in Chapter 4 of \cite{KL}. A similar result was obtained in \cite{AveFraJarVol} for the process $\{\xi_t^n;t \in [0,T]\}$, but before stating this result, we need some notation. Let us define $v^{\pm}: [0,1] \to \bb R$ as
\begin{equation*}
v^{\pm}(\rho) = \int c^{\pm}(\eta;x) \nu_\rho(d\eta).
\end{equation*}
Notice that $v^{\pm}$ do not depend on $x$. Since we have assumed that $c$ is local, $v^{\pm}$ do not depend on $n$ either. Define then $v(\rho) = v^+(\rho) - v^-(\rho)$. The value of $v(\rho)$ can be interpreted as the ``mean-field'' speed of the walk $\{x_t^n; t \in [0,T]\}$ in an environment of density $\rho$, but we point out that this far from clear under which conditions we can assume that this mean-field speed is a good approximation for the real speed of the walk. The following propositions are the main results in \cite{AveFraJarVol}.

\begin{proposition}
\label{p2} 
With respect to $\bb P_n$, 
\begin{equation*}
\lim_{n \to \infty} \hat{\pi}_t^n(dx) = \hat{u}(t,x) dx
\end{equation*}
in law with respect to the $J_1$-Skorohod topology of $\mc D([0,T], \mc M^+(\bb T))$, where the density $\{\hat{u}(t,x); t \in [0,T], x \in \bb T\}$ is the solution of the equation
\begin{equation*}
\begin{cases}
\partial_t \hat{u}(t,x) & = \Delta \hat{u} (t,x) + v(\hat{u}(t,0))\partial_x \hat{u}(t,x)\\
\hat{u}(0,\cdot) &= u_0(\cdot).
\end{cases}
\end{equation*}
\end{proposition}

Let $\{f(t); t \in [0,T]\}$ be the solution of the differential equation 
\begin{equation*}
\begin{cases}
f'(t) &= v(u(t,f(t)))=v(\hat{u}(t,0))\\
f(0) &=0,
\end{cases}
\end{equation*}
with $u$ from Proposition \ref{p1}. The densities $u$ and $\hat u$ are related by the identity $\hat{u}(t,x) = u(t,f(t)+x)$ for any $t \in [0,T]$ and any $x \in \bb T$. In fact, we have the following law of large numbers for $\{x_t^n; t \in [0,T]\}$.

\begin{proposition}
\label{p3}
With respect to $\bb P_n$, 
\begin{equation*}
\lim_{n \to \infty} x_t^n = f(t)
\end{equation*}
in distribution with respect to the $J_1$-Skorohod topology on  $\mc D([0,T]; \bb T)$.
\end{proposition}

\section{Main results: large deviations}
\label{s3}
Propositions \ref{p1} and \ref{p3} can be understood as a functional law of large numbers for the pair of processes $\{(\pi_t^n, x_t^n); t \in [0,T]\}$. Our aim is to establish a large deviation principle for the process $\{x_t^n; t \in [0,T]\}$, Theorem \ref{t1} below.

\subsection{Topological considerations}\label{topo}
Let us notice that the $J_1$-Skorohod topology coincides with the uniform topology when restricted to the space of continuous functions. This topology is not the only one with this property. Indeed, in the original work of Skorohod \cite{S}, four different topologies are introduced on the space $\mc D([0,T]; \mc E)$ with this property, and such that the space $\mc D([0,T]; \mc E)$ is Polish with respect to these topologies. Let us recall the decomposition $x_t^n = \Pi(N_t^{n,+}-N_t^{n,-})$. Since $N_t^{n,+}+N_t^{n,-}$ is just a standard Poisson process speeded-up by $n$, an immediate corollary of Proposition \ref{p3} is that
\begin{equation*}
\lim_{n \to \infty} \frac{N_t^{n,\pm}}{n} = \frac{t \pm \hat{f}(t)}{2}
\end{equation*}
in distribution with respect to the $J_1$-Skorohod topology on  $\mc D([0,T]; \bb R)$, where $\{\hat{f}(t); t \in [0,T]\}$ is the canonical lifting of $\{f(t); t \in [0,T]\}$ from $\bb T$ to $\bb R$. In fact, the convergences of the processes $\{x_t^n; t \in [0,T]\}$ and $\{\frac{1}{n}N_t^{n,+}; t \in [0,T]\}$ are equivalent, once we have the law of large numbers for the standard Poisson process. Notice that the process $\{N_t^{n,+}; t \in [0,T]\}$ is {\em increasing}. Therefore, maybe the $J_1$-Skorohod topology is not the most suitable one. It turns out that in order to exploit the fact that $\{N_t^{n,+}; t \in [0,T]\}$ is increasing, we can use the {\em weak topology} in the following way. Let us denote by $\omega_{\pm}^n(dt)$ the measure on  $[0,T]$ defined by $\omega^n_{\pm}((s,t]) = \frac{1}{n}(N_t^{n,\pm} -N_s^{n,\pm})$ for any $s<t \in [0,T]$. Then, convergence of $\{\frac{1}{n}N_t^{n,\pm}; t \in [0,T]\}$ to $\{\frac{1}{2}(t\pm\hat{f}(t)); t \in [0,T]\}$ is equivalent to convergence of the sequence of positive Radon 
measures 
$\{\omega^n_{\pm}; n \in \bb N\}$ to the measure $\frac{1}{2}(1\pm\hat{f}'(t))dt$, with respect to the weak topology of $\mc M^+([0,T])$. We will adopt this last point of view. Notice that in order to recover the process $\{x_t^n; t \in [0,T]\}$, we need both processes $\{N_t^{n,\pm};t \in [0,T]\}$, or equivalently, both measures $\{\omega^n_{\pm}\}$. Therefore, if needed, we can consider the process $\{x_t^n; t \in [0,T]\}$ as an element of the space $\mc M^+([0,T]) \times \mc M^+([0,T])$ equipped with the weak topology. The main advantage of this point of view is the characterization of compact sets, which is very simple on  $\mc M^+([0,T])$: a set  $\mc K \subseteq \mc M^+([0,T])$ is relatively compact if and only if $\sup_{\mu \in \mc K} \mu([0,T])<+\infty$. Further topological considerations will be introduced at the occurrence in the proof of the large deviation principle.

\subsection{Large deviation principle}
We start by recalling what a large deviation principle is. 
Since we are going to state several large deviation principles, let us define it in full generality.
Let $\mc E$ be a Polish space. Given a function $\mc I: \mc E \to [0,\infty]$, we call it {\em rate function} if it is lower semi-continuous, that is, the set $\{x \in \mc E; \mc I(x) \leq M\}$ is closed for any $M \in [0,\infty)$. We say that the rate function $\mc I$ is {\em good} if the sets $\{x \in \mc E; \mc I(x) \leq M\}$ are {\em compact} for any $M \in [0,\infty)$. A sequence $\{X_n; n \in \bb N\}$ of $\mc E$-valued random variables defined in some probability space $(E,\mc F,P)$ satisfies a large deviation principle with good rate function $\mc I$ if
\begin{itemize}
\item[i)] for any open set $\mc A \subseteq \mc E$,
\begin{equation*}
\varliminf_{n \to \infty} \frac{1}{n} \log P(X_n \in \mc A) \geq -\inf_{x \in \mc A} \mc I(x),
\end{equation*}
\item[ii)] for any closed set $\mc C \subseteq \mc E$,
\begin{equation*}
\varlimsup_{n \to \infty} \frac{1}{n} \log P(X_n \in \mc C) \leq -\inf_{x \in \mc C} \mc I(x).
\end{equation*}
\end{itemize}

\subsection{The initial distribution of particles}
In Section \ref{s1.6}, we saw that in order to obtain the hydrodynamic limit of the environment process, the initial distribution of particles must be associated to some profile $u_0$. It turns out that in order to obtain a large deviation principle for the environment process, it is necessary (but far from sufficient) to understand the large deviations of the initial distribution of particles. Let $u_0$ be a given initial profile. For simplicity, we assume that $u_0$ is continuous and that there exists $\epsilon >0$ such that $u_0 \in [\epsilon, 1-\epsilon]$. Recall the definition of the measures $\{\nu_{u_0}^n; n \in \bb N\}$ given in Section \ref{s1.6}. With respect to $\{\nu_{u_0}^n; n \in \bb N\}$, the empirical measure $\pi_0^n$ converges in distribution to the measure $u_0(x)dx$, and a large deviation principle for the sequence $\{\pi_0^n; n \in \bb N\}$ is not difficult to obtain. Recall that we consider $\pi_0^n$ as an element in $\mc M_{0,1}^+(\bb T)$. Let $v_0(x)dx$ be an element of $\mc M_{0,1}^+(\bb T)$. This imposes the restriction $0 \leq v_0(x) \leq 
1$ for any $x \in \bb T$. Define
\begin{equation}\label{RateFn1}
h(v_0|u_0) := \int_{\bb T} \Big\{ u_0(x) \log\Big(\tfrac{u_0(x)}{v_0(x)}\Big) + (1-u_0(x))\log\tfrac{1-u_0(x)}{1-v_0(x)}\Big\} dx.
\end{equation}
The large deviations of the initial distribution of particles is given by the following proposition (see e.g. \cite{KL}, Lemma 5.2, Chapter 10). 
\begin{proposition}
\label{p4} The sequence $\{\pi_0^n; n \in \bb N\}$ satisfies a large deviation principle with respect to the weak topology on  $\mc M_{0,1}^+(\bb T)$ with rate function $h$.
\end{proposition}

\subsection{Large deviation principle for the environment}
A large deviation principle for the process $\{\pi_t^n; t \in [0,T]\}$ has been obtained in \cite{KOV}. 
Let us recall this result. For $H: [0,T] \times \bb T \to \bb R$ of class $\mc C^{1,2}$ and $\{\pi_t; t \in [0,T]\}$ in $\mc D([0,T]; \mc M_{0,1}^+(\bb T))$, define
\begin{equation}\label{RateFn2}
\begin{split}
J(H;\pi) := \pi_T(H_T)
		&- \pi_0(H_0) - \int_0^T \pi_t\big(\partial_t H_t + 2\Delta H_t\big)  dt\\
		&\quad-\int_0^T \int \big( \nabla H_t(x)\big)^2 \pi_t(x)\big(1-\pi_t(x)\big) dxdt,
\end{split}
\end{equation}
and set
\begin{equation*}
\mc I_{\mathrm{ex}}(\pi) := h(\pi_0|u_0) + \nop_{H \in \mc C^{1,2}} J(H;\pi).
\end{equation*}
The following proposition is the main result in \cite{KOV}.
\begin{proposition}
\label{p5}
The process $\{\pi_t^n; t \in [0,T]\}$ satisfies a large deviation principle with good rate function $\mc I_{\mathrm{ex}}$ with respect to the $J_1$-Skorohod topology on the path space $\mc D([0,T]; \mc M^+_{0,1}(\bb T))$.
\end{proposition}

\subsection{Large deviations for the random walk}
For each function $x:[0,T] \to \bb T$ of finite variation with $x_0=0$ and each $\pi: [0,T] \to \mc M_{0,1}^+(\bb T)$ c\`adl\`ag, let us define
\begin{align}\label{RateFn3}
 \mc I_{\mathrm{rw}}(x|\pi) &= 
 \int_0^T \Big\{ a_{x,\pi}(t) x'_t -\sum_{z=\pm} v^z(\pi_t(x_t))(e^{za_{x,\pi}(t)}-1)\Big\} dt,	\text{ where} \\
a_{x,\pi}(t) &= 
\begin{cases}
\log \frac{x'_t+\sqrt{(x'_t)^2+4v^+(\pi_t(x_t))v^-(\pi_t(x_t))}}{2v^+(\pi_t(x_t))}, &v^+(\pi_t(x_t))v^-(\pi_t(x_t))>0,\\
\log\frac{|x_t'|}{v^+(\pi_t(x_t))},&\hspace{-4em}v^+(\pi_t(x_t))v^-(\pi_t(x_t))=0,\ x_t'>0, \\
-\log\frac{|x_t'|}{v^-(\pi_t(x_t))},&\hspace{-4em}v^+(\pi_t(x_t))v^-(\pi_t(x_t))=0,\ x_t'<0, \\
-\infty,&\hspace{-4em}v^+(\pi_t(x_t))v^-(\pi_t(x_t))=0,\ x_t'=0,\ v^+(\pi_t(x_t))>0, \\
\infty,&\hspace{-4em}v^+(\pi_t(x_t))v^-(\pi_t(x_t))=0,\ x_t'=0,\ v^+(\pi_t(x_t))=0,
\end{cases}
\end{align}
if $x$ is absolutely continuous and $x \mapsto \pi_t(x)$ is continuous at $x_t$ for a.e.~$t \in [0,T]$. Otherwise, or if one of the three integrals
\begin{align}
&\int_0^T |x'_t|\log^+|x_t'| \,dt  \quad\text{or}\quad \int_0^T (x'_t)^z\log^+\left((v^z(\pi_t(x_t)))^{-1}\right) \,dt , \ z=\pm,
\end{align}
is infinite, then $\mc I_{\mathrm{rw}}(x|\pi) = \infty$, where $f^+=\max(f,0)$ and $f^-=\max(-f,0)$ are the positive and negative part of a function (note that due to a collision of notation, $v^+$ and $v^-$ are separate functions, not positive and negative part of some function $v$).

Our main result is the following.

\begin{theorem}
\label{t1}
The sequence $\{x_t^n; t \in [0,T]\}_{n \in \bb N}$ satisfies a large deviation principle with good rate function
\begin{equation*}
\mc I(x) = \inf_{\pi} \big\{ \mc I_{\mathrm{rw}}(x|\pi) + \mc I_{\mathrm{ex}}(\pi)\big\}.
\end{equation*} 
\end{theorem}

Actually, this result will be a consequence of a large deviation principle for the {\em pair} $\{(\pi_t^n; x_t^n); t \in [0,T]\}$.

\begin{theorem}
\label{t2}
The sequence $\{(\pi_t^n; x_t^n); t \in [0,T]\}$ satisfies a large deviation principle with good rate function $\mc I_{\mathrm{rw}}(x|\pi)+\mc I_{\mathrm{ex}}(\pi)$.
\end{theorem}

The rest of the paper is devoted to the proof of Theorems \ref{t1} and \ref{t2}.

\section{Tilting measures and exponential martingales}\label{Martingales} 

According to Donsker-Varadhan approach to large deviations \cite{DV}, in order to show a large deviation principle, it is necessary to construct a sufficiently rich family of exponential martingales. The rough idea which will be clear along the proof is that these exponential martingales will be used to tilt the original distribution of the system in consideration, in such a way that atypical events become typical under the tilted distribuion.
Let us introduce the family of martingales relevant for our scope.
Recall equation \eqref{jointGen} and let $F: \Omega_n \times \bb T_n \times [0,T] \to \bb R$ be differentiable in the time variable. Then, the process
\begin{equation}\label{F}
\exp\Big\{ F_t(\eta_t^n; x_t^n) -F_0(\eta_0^n; x_0^n) - \int_0^t e^{-F_s(\eta_s^n; x_s^n)} \big(\partial_s + L_n\big) e^{F_s(\eta_s^n; x_s^n)} ds\Big\}
\end{equation}
is a positive martingale of unit expectation (see e.g. \cite{KL}, Lemma 5.1 in Appendix 1).
It turns out that there are two types of relevant functions for the large deviations problem. Let $a: [0,T] \to \bb R$ be a continuously differentiable function. Taking $F_t(\eta;x) = n a(t)x$ in \eqref{F}, we see that the process $\{\mc M_t^{a,n}; t \in [0,T]\}$ given by
\begin{equation}\label{Ma}
\frac{1}{n}\log \mc M_t^{a,n} = a(t) x_t^n - a(0) x_0^n -\int_0^t \Big(a'(s) x_s^n + \sum_{z=\pm 1} c^{z}(\eta_s^n; x_s^n) \big(e^{za(s)}-1\big)\Big)ds
\end{equation}
is a positive martingale with unit expectation. Notice that by definition, $a(0)x_0^n \equiv 0$. Notice as well that integrating by parts, we see that
\begin{equation*}
a(T)x_T^n- \int_0^T a'(t)x_t^n dt = \int_0^T a(t) \omega^n(dt).
\end{equation*}
Therefore, in a sense, knowing $\mc M_T^{a,n}$ for every $a$, we know $\{x_t^n; t \in [0,T]\}$. 

The second type of function that plays a role in the derivation of a large deviation principle is the following. Let $H: [0,T] \times \bb T \to \bb R$ of class $\mc C^{1,2}$, that is, once continuously differentiable in time and twice continuously differentiable in space. 
Let us define 
\begin{equation*}
\nabla_{x,y}^n H_t := n^2 \int \big(\delta_y^n(z)-\delta_x^n(z)\big)H_t(z) dz,
\end{equation*}
\begin{equation*}
\Delta_n H_t(x) := n\sum_{\substack{ y \in \bb T_n\\y \sim x}} \nabla_{x,y}^n H_t.
\end{equation*}
It is not difficult to check that for $x \in \bb T_n$, $y=x+\frac{1}{n}$, the function $\nabla_{x,y}^nH_t$ is a discrete approximation of the gradient $\nabla H_t(x)$, and that $\Delta_n H_t(x)$ is a discrete approximation of the Laplacian $\Delta H_t(x)$. We extend the definition of $\Delta_n H_t$ to $\bb T$ by taking linear interpolations.
Taking $F_t(\eta;x) = n \pi_t^n(H_t)$ in \eqref{F}, we see that the process $\{\mc M_t^{H,n}; t \in [0,T]\}$ given by
\begin{equation}\label{MH}
\frac{1}{n}\log \mc M_t^{H,n} = \pi_t^n(H_t) - \pi_0^n(H_0) - \!\int_0^t \!\!\!\Big\{\pi_s^n(\partial_sH_s) +\frac{2}{n} \!\!\sum_{x \in \bb T_n}\!\!  \eta_s^n(x) \Delta_n H_s(x) \Big\}ds - \mc Q_t^n(H),
\end{equation}
where
\begin{equation*}
\mc Q_t^n(H) = 2\int_0^t n \sum_{x \sim y} \eta_s^n(x)\big(1-\eta_s^n(y)\big)\psi\Big(\frac{1}{n}\nabla_{x,y}^nH_s\Big)ds 
\end{equation*}
and $\psi(u) = e^u-u-1$, is a positive martingale with unit expectation\footnote{
Notice that we are making an abuse of notation, using the same superscript structure for $\mc M_t^{a,n}$ and $\mc M_t^{H,n}$. Later on we will introduce some more efficient way to handle multiple indices.
}.
Since we are assuming that  $H$ is of class $\mc C^{1,2}$ we can write
\begin{equation*}
\frac{1}{n} \!\!\sum_{x \in \bb T_n} \!\! \eta_s^n(x) \Delta_n H_s(x)  = \pi_s^n(\Delta H_s)  + \mc R_s^n(H),
\end{equation*}
where the error term $\mc R_s^n(H)$ is bounded by a  function of the form $r_n(H)$, depending only on the modulus of continuity of $\Delta H$ in $\bb T \times [0,T]$ and converging to $0$ as $n$ tends to $\infty$. Since the jumps of the environment and the particle are a.s.~disjoint, the martingales $\{\mc M_t^{a,n}; t \in [0,T]\}$, $\{\mc M_t^{H,n}; t \in [0,T]\}$ are orthogonal, in the sense that the process $\{\mc M_t^{a,n}\mc M_t^{H,n}; t\in [0,T]\}$ is also a positive martingale with unit expectation.

\section{The superexponential estimate}
\label{s4}
One of the main challenges in order to prove a large deviation principle in the context of interacting particle systems, is to show that local functions of the dynamics, when averaged over space and time, can be expressed as functions of the empirical measure plus an error which is superexponentially small. Let us explain what the superexponential estimate is in the case of the simple exclusion process (that is, our environment process). In order to do this, we need some notation. Let $f: \Omega \to \bb R$ be a local function. Recall the convention about how to project $f$ into $\Omega_n$. Define $\bar{f}(\rho) = \int f d\nu_\rho$ for $\rho \in [0,1]$. For $\epsilon \in (0,\frac{1}{2})$ and $x \in \bb T$, let us define $\iota_\epsilon(x) = \frac{1}{\epsilon} \mathbf{1}((x,x+\epsilon]))$. When $x=0$, we just write $\iota_\epsilon$ instead of $\iota_\epsilon(0)$. 
The following lemma is stated in \cite{KOV}, Theorem 2.1.

\begin{lemma}[Superexponential estimate]
\label{p8}
Let $H: [0,T] \times \bb T \to \bb R$ be a continuous function. Let us define
\begin{equation*}
\mc R_t^{n, \epsilon}(H) = \int_0^t \frac{1}{n} \sum_{x \in \bb T_n} \big\{ \tau_x f(\eta_s^n) - \bar{f} \big(\pi_s^n(\iota_\epsilon(x))\big) \big\} H_s(x) ds.
\end{equation*}
Then, for any $\delta >0$, and any $t \in [0,T]$,
\begin{equation*}
\varlimsup_{\epsilon \to 0} \varlimsup_{n \to \infty} \frac{1}{n} \log \bb P_n\big(\big|\mc R_t^{n,\epsilon}(H)\big|>\delta \big) =-\infty.
\end{equation*}
\end{lemma}

This superexponential estimate is used in \cite{KOV} with two purposes. First, to express $\mc Q_t^n(H)$ (recall the definition of the martingale $\{\mc M_t^{H,n}; t \in [0,T]\}$in \eqref{MH}) as a function of $\{\pi_t^n; t \in [0,T]\}$ plus an error that is superexponentially small. And second, in order to obtain the hydrodynamic limit of suitable perturbations of the exclusion dynamics. Notice that, as a consequence, we can express $\mc M_t^{H,n}$ (more precisely, $\frac{1}{n} \log \mc M_t^{H,n}$) as a function of $\{\pi_t^n; t \in [0,T]\}$ plus a superexponentially small error. Recalling \eqref{Ma}, we see that in order to express $\mc M_t^{a,n}$ as a function of $\{(\pi_t^n, x_t^n); t \in [0,T]\}$, we need to express
\begin{equation*}
\int_0^t c^\pm(\eta_s^n; x_s^n)\big(e^{\pm a(s)}-1\big) ds
\end{equation*}
as a function of these two processes. The superexponential estimate does not apply for two reasons. First, there is no spatial average. Second, the position at which we measure the local function $c^\pm$ changes with time (since it follows the location of the random walk). In \cite{JLS} and in the context of the tagged particle problem, both problems were overcome by considering the environment as seen from the walk, $\{\xi_t^n; t \in [0,T]\}$. Notice that in terms of the process $\{\xi_t^n; t \in [0,T]\}$, the integral in question is given by
\begin{equation*}
\int_0^t c^\pm(\xi_s^n) \big(e^{\pm a(s)}-1\big) ds.
\end{equation*}

In this section, our objective will be to show the following superexponential estimate.

\begin{lemma}[Local superexponential estimate]
\label{l1}
Let $f: \Omega_n \to \bb R$ be a local function. Then,
\begin{equation}
\label{ec42}
\varlimsup_{\epsilon \to 0} \varlimsup_{\vphantom{0}n \to \infty} \frac{1}{n} \log \bb P_n\Big(\Big|\int_0^t \big\{f(\xi_s^n) - \bar{f}\big(\hat{\pi}_s^n(\iota_\epsilon)\big)\big\} ds \Big| > \delta \Big) = -\infty
\end{equation}
for any $\delta>0$ and any $t \in [0,T]$.
\end{lemma}

To make the exposition clear, the proof will be divided in various steps. Before starting the proof, we introduce some notations and conventions. Let us write
\begin{equation}\label{abbrev1}
\mc W_f^\ell(\xi) = f(\xi) - \bar{f} \Big(\frac{1}{\ell} \sum_{x=1}^\ell \xi(x)\Big).
\end{equation}
With this notation, the integral in the local superexponential lemma is equal to
\begin{equation}\label{abbrev2}
\int_0^t \mc W_{f}^{\epsilon n}(\xi_s^n) ds.
\end{equation}
For simplicity, we assume that the support of $f$ is contained on $\{1,\dots,\ell_0\}$ for some $\ell_0 \in \bb N$. In that case, $\supp(\mc W_f^\ell) = \{1,\dots,\ell\}$ for any $\ell \geq \ell_0$. We will indistinctly denote by $\Lambda_\ell$ the sets $\{1,\dots,\ell\} \subseteq \bb Z$ and $\{\frac{1}{n},\dots,\frac{\ell}{n}\} \subseteq \Omega_n$.

\subsection{Reduction to a variational problem}
\label{s3.1}
In this section we reduce the proof of the superexponential estimate to a variational problem involving the generator of the dynamics.
Let us start by introducing an elementary estimate, whose check is left to the reader, which will be used several times.

\begin{lemma}
\label{p10}
For any positive numbers $a_1,\dots,a_\ell$,
\begin{equation*}
\log\{a_1+\dots+a_\ell\} \leq \max_{1\leq j\leq \ell} \log a_j + \log \ell.
\end{equation*}
\end{lemma}

Using this lemma, we see that for any random variable $X$, 
\begin{equation*}
\begin{split}
\log P(|X|>\delta) &\leq \log\{P(X>\delta) + P(X<-\delta)\}\\ 
		&\leq \max\{\log P(X>\delta),\log P(-X>\delta)\}+\log 2.
\end{split}
\end{equation*}
Therefore, in order to show \eqref{ec42}, it is enough to show that
\begin{equation}
\label{ec47}
\varlimsup_{\epsilon \to 0} \varlimsup_{\vphantom{0}n \to \infty} \frac{1}{n} \log \bb P_n \Big( \pm \int_0^t \mc W_f^{\epsilon n}(\xi_s^n) ds > \delta\Big) =-\infty.
\end{equation}
Therefore, we get rid of the absolute value in \eqref{ec42}. This has several advantages as it will be clear soon. 
By the exponential Chebyshev's inequality, for any random variable $X$ and any $\gamma >0$ we have that
\begin{equation*}
\frac{1}{n} \log P(\pm X >\delta) \leq \frac{1}{n} \log \frac{E[e^{\pm\gamma n X}]}{e^{\pm \gamma n \delta}} = \frac{1}{n} \log E[e^{\gamma n X}] -\gamma \delta.
\end{equation*}
Therefore, it is enough to show that 
\begin{equation}\label{donno}
\nop_{\gamma} \varlimsup_{\epsilon \to 0} \varlimsup_{\vphantom{0}n \to \infty} \frac{1}{n} \log \bb E_n\Big[ \exp\Big\{ \pm \gamma n \int_0^t \mc W_f^{\epsilon n}(\xi_s^n)ds\Big\}\Big] < +\infty,
\end{equation}
since in that case, calling this supremum $\kappa$,
\begin{equation*}
\varlimsup_{\epsilon \to 0} \varlimsup_{\vphantom{0}n \to \infty} \frac{1}{n} \log \bb P_n\Big(\pm\int_0^t \mc W_f^{\epsilon n} (\xi_s^n) ds > \delta\Big) \leq \kappa -\gamma \delta
\end{equation*}
for any $\gamma>0$ and sending $\gamma$ to infinity, \eqref{ec47} follows. Since $- \mc W_f^{\epsilon n} = \mc W_{-f}^{\epsilon n}$, from now on we omit the $\pm$ in \eqref{donno}.

The next step is to put the process in near-equilibrium distribution. Fix $\rho \in (0,1)$ and let us denote by $\bb P_n^{\rho}$ the distribution of the process $\{\xi_t^n; t \in [0,T]\}$ with initial distribution $\nu_\rho$ (or equivalently, the process $\{(\eta_t^n,x_t^n); t \in [0,T]\}$ with initial distribution $\nu_\rho \otimes \delta_0$), and let $\bb E_n^\rho$ be the expectation with respect to $\bb P_n^\rho$. The actual value of $\rho$ will not be very important. Notice that $\nu_\rho$ is {\em not} stationary under the evolution of $\{\xi_t^n; t \in [0,T]\}$, but it is indeed close to stationarity in a sense to be specified below. By the Markov property, $\frac{d \bb P_n}{d \bb P_n^\rho}=\frac{d \mu^n}{d \nu_\rho}$. Moreover, since $\nu_\rho(\eta) \geq \min\{\rho,1-\rho\}$ for any $\eta \in \Omega_n$ (in fact, the worst configurations are $\eta(x) \equiv 0 \text{ or } 1$), we conclude that there exists a constant $K_0 = K_0(\rho)$ such that $\|\frac{d \mu^n}{d \nu_\rho}\|_\infty \leq K_0^n$ for any $n 
\in \bb N$. In particular, for any function $F \geq 0$, 
\begin{equation*}
\bb E_n[F] = \bb E_n^\rho\big[\tfrac{d \nu^n}{d \nu_\rho}F\big] \leq K_0^n \bb E_n^\rho[F]. \end{equation*}
Therefore, from \eqref{donno}, we get
\begin{equation}\label{donno2}
\frac{1}{n} \log \bb E_n\Big[ \exp\Big\{ \gamma n \int_0^t \mc W_f^{\epsilon n}(\xi_s^n)ds\Big\}\Big] 
		\leq \frac{1}{n} \log \bb E_n^\rho\Big[ \exp\Big\{ \gamma n \int_0^t \mc W_f^{\epsilon n}(\xi_s^n)ds\Big\}\Big] + K_0,
\end{equation}
and it is enough to consider the case $\mu^n = \nu_\rho$. The nowadays classical argument of Varadhan (see Lemma A1.7.2 on page 336 of \cite{KL}) to estimate exponential expectations as in the r.h.s. of \eqref{donno}, combines Feynman-Kac's formula with the variational formula for the largest eigenvalue of the operator $\mc L_n+ \mc W_f^{\epsilon n}$, to get the bound
\begin{equation}
\label{ec53}
\frac{1}{n} \log \bb E_n^\rho\Big[ \exp\Big\{ \gamma n \int_0^t \mc W_f^{\epsilon n}(\xi_s^n)ds\Big\}\Big] 
		\leq t \nop_g\big\{ \gamma \<\mc W_f^{\epsilon n} , g^2\> +\tfrac{1}{n} \<g, \mc L_n g\>\big\},
\end{equation}
where $\<\cdot,\cdot\>$ denotes the inner product in $L^2(\nu_\rho)$, the supremum runs over functions $g: \Omega_n \to \bb R$ such that $\<g,g\>=1$ and $\mc L_n$ is the generator of the process $\{\xi_t^n; t \in [0,T]\}$ in equation \eqref{EPGen}. This variational problem will be the starting point of the next step of the proof.

\subsection{Some properties of $\<g, \mc L_n g\>$}
Define, for $g: \Omega_n \to \bb R$ and $x,y \in \bb T_n$,
\begin{equation*}
\mc D^{x,y}(g) = \frac{1}{2} \int \big(g(\eta^{x,y})-g(\eta)\big)^2 d\nu_\rho,
\end{equation*}
and define $\mc D(g) = \sum_{x \sim y} \mc D^{x,y}(g)$. Notice that $\<g,-L^{ex}_ng\>=n^2\mc D(g)$, that is, $n^2 \mc D(g)$ is the {\em Dirichlet form} associated to the exclusion process $\{\eta_t^n; t \in [0,T]\}$ identified by the generator in equation \eqref{exGen}. The following proposition was proved in \cite{AveFraJarVol}, see Lemma 2.2 therein.

\begin{proposition}
\label{p11}
There exists a constant \footnote{Note that under the assumption $c^++c^- \equiv 1$, one can take $K_1=1$.}$K_1$ such that $\<g,\mc L_n g\> \leq -n^2 \mc D(g) + K_1 n$ for any function $g: \Omega_n \to \bb R$ such that $\<g,g\>=1$.
\end{proposition}

The intuition behind this proposition is the following. The quantity $\<g,\mc L_n g\>$ measures the entropy production rate, and if $\nu_\rho$ were invariant, it should be negative. Since $\nu_\rho$ is invariant under the dynamics of the environment, entropy can grow only due to the motion of the random walk. Since the random walk jumps about $n$ times on a fixed time interval, the entropy of the distribution of the process with respect to $\nu_\rho$ should grow with time at most linearly in $n$.

The following simple observation, which we state as a proposition,  will be useful in what follows.

\begin{proposition}
\label{p12}
For any $x,y \in \bb T_n$, the function 
\begin{equation*}
g \mapsto \int \big( \sqrt{g(\eta^{x,y})}-\sqrt{g(\eta)}\big)^2 \nu_\rho(d\nu)
\end{equation*}
is convex. In particular, $g\mapsto \mc D(\sqrt g)$ is convex.
\end{proposition}

\subsection{The one-block estimate}
\label{s3.3}
In the previous section, we have reduced the proof of \eqref{ec42} to the variational problem
\begin{equation*}
\nop_{\gamma}\varlimsup_{\epsilon \to 0} \varlimsup_{\vphantom{0} n \to \infty} \sup_g\big\{ \gamma \<\mc W_f^{\epsilon n} , g^2\> +\tfrac{1}{n} \<g, \mc L_n g\>\big\} <+\infty.
\end{equation*}
Following the original idea of Guo, Papanicolaou and Varadhan \cite{GPV}, \cite{DV}, \cite{KOV}, it is convenient to break this variational problem into two pieces. The first one is what is known as the {\em one-block estimate}. In this estimate, the macroscopically small box of size $\epsilon n$ is replaced by a microscopically large box of size $\ell$, and it corresponds to the following lemma:

\begin{lemma}[One-block estimate]
\label{l2} 
\begin{equation*}
\nop_{\gamma} \varlimsup_{\ell \to \infty} \varlimsup_{n \to \infty} \nop_g \big\{ \gamma \<\mc W_f^\ell, g^2\> + \frac{1}{n}\<g, \mc L_n g\>\big\}<+\infty.
\end{equation*}
\end{lemma}
\begin{proof}[Proof of Lemma \ref{l2}]
By Proposition \ref{p11}, 
\begin{equation*}
\nop_g \big\{ \gamma \<\mc W_f^\ell, g^2\> + \frac{1}{n}\<g, \mc L_n g\>\big\} \leq 
		K_1+ \nop_g \big\{ \gamma \<\mc W_f^\ell, g^2\> -n\mc D(g)\big\}.
\end{equation*}
For the ones acquainted with the theory of hydrodynamic limits well, the supremum on the right-hand side of this inequality is basically the one appearing in Eq.~5.4.1 of \cite{KL}, and the proof there applies to our situation with essentially no changes. For the ones who are not familiar with hydrodynamic limits, we include a somehow simpler proof. Let us define $\mc F_\ell = \sigma\{\xi(x); x \in \Lambda_\ell\}$, where the set $\Lambda_\ell = \{\frac{1}{n},\dots, \frac{\ell}{n}\}$ was defined above. Notice that for any function $g$, $\mc D(|g|) \leq \mc D(g)$, while $g^2=|g|^2$. Therefore, we can restrict the supremum above to {\em non-negative} functions $g: \Omega_n \to \bb R$  such that $\<g,g\>=1$. Let us define
\begin{equation*}
\mc D_\ell(g) = \sum_{\mathclap{\smash[b]{\substack{x,y \in \Lambda_\ell\\x \sim y}}}} \mc D^{x,y}(g).
\end{equation*}
For a given non-negative function $g$ with $\<g,g\>=1$, let us define $g_\ell = E_{\nu_\rho}[g^2|\mc F_\ell]^{\smash{\frac{1}{2}}}$. By definition, $\<\mc W_f^\ell,g_\ell^2\>=\<\mc W_f^\ell,g^2\>$, while by convexity, $\mc D_\ell(g_\ell) \leq \mc D_\ell(g) \leq \mc D(g)$. Therefore, 
\begin{equation*}
\gamma \<\mc W_f^\ell, g^2\> - n \mc D(g) \leq \gamma \<\mc W_f^\ell, g_\ell^2\> - n \mc D_\ell(g_\ell)
\end{equation*}
and it is enough to show that
\begin{equation*}
\nop_\gamma \varlimsup_{\ell \to \infty} \varlimsup_{n \to \infty} \nop_g \big\{\gamma \<\mc W_f^\ell, g^2\>-n \mc D_\ell(g)\big\} <+\infty,
\end{equation*}
where now the supremum runs over functions $g:\Omega_n \to \bb R$ such that $\<g,g\>=1$ and such that $\supp(g) \subseteq \Lambda_\ell$. Notice that on the supremum above, the only dependence on $n$ is on the constant in front of $\mc D_\ell(g)$. Moreover, the variational problem is a finite-dimensional one ($2^\ell$-dimensional). In particular, $g$ lives in a compact space (topology does not matter here, because all the metrics are equivalent in finite-dimensional spaces). Therefore, for each $n$, there exists a function $g^n$ for which the supremum is attained. For $g \equiv 1$, $\gamma \<\mc W_f^\ell,g^2\>-n \mc D_\ell(g)=0$. Therefore, the supremum is greater or equal than 0. Therefore, $\mc D_\ell(g^n) \leq \frac{\gamma}{n} \|f\|_\infty$, and in particular $\mc D_\ell(g^n)$ tends to $0$ as $n$ tends to $\infty$. Let $n'$ be a subsequence such that $g^n$ converges to some limit $g^\infty$. Since $g \mapsto \mc D_\ell(g)$ is convex, it is also lower semi-continuous. Therefore, we have 
that $\mc D_\ell(g^\infty) \leq \varliminf_{n'} \mc D_\ell(g^{\smash{n'}})=0$. We have just showed that 
\begin{equation*}
\varlimsup_{n \to \infty} \nop_g \big\{\gamma \<\mc W_f^\ell, g^2\> - n \mc D_\ell(g)\big\} = \gamma \<\mc W_f^\ell, \hat{g}^2\>
\end{equation*}
for some function $\hat{g}:\Omega_n \to \bb R$ satisfying $\<g,g\>=1$, $\supp(g) \subseteq \Lambda_\ell$ and $\mc D_\ell(g)=0$. Let us identify $\{0,1\}^{\Lambda_\ell}$ with $\Omega_\ell$, where we forget about the periodic boundary condition. Recall the definition of the spaces $\Omega_{k,\ell}$ given in Section \ref{s1.1}. By the irreducibility of the exclusion process, $\mc D_\ell(\hat g)=0$ implies that $\hat g$ is constant on each of the spaces $\Omega_{k,\ell}$, $k=0,1,\dots,\ell$. On the set $\Omega_{k,\ell}$, 
\begin{equation*}
\mc W_f^\ell(\eta) = f(\eta) -\bar{f}\big(\tfrac{k}{\ell}\big).
\end{equation*}
Therefore, there exists a sequence of positive numbers $\{p(0),\dots,p(\ell)\}$ such that $\sum_k p(k)=1$ and 
\begin{equation*}
\<\mc W_f^\ell, \hat{g}^2\> = \sum_{k=0}^\ell p(k) \big\{\bar f(k;\ell) - \bar{f}\big(\tfrac{k}{\ell}\big)\big\},
\end{equation*}
where $\bar f(k;\ell) = \int fd\nu_{k,\ell}$. We have thus reduced the proof of the one-block estimate to proving that
\begin{equation*}
\varlimsup_{\ell \to \infty} \nop_{1\leq k\leq \ell}\big|\bar{f}(k;\ell)-\bar{f}\big(\tfrac{k}{\ell}\big)\big| =0.
\end{equation*}
This limit is equal to $0$ in view of Prop.~3.1 in \cite{GJ}, known in the literature as the {\em equivalence of ensembles}. This finishes the proof of Lemma \ref{l2}.
\end{proof}

\subsection{The two-blocks estimate}
\label{s3.4}
In view of Lemma \ref{l2}, in order to complete the proof of Lemma \ref{l1}, it is enough to show the following.

\begin{lemma}[Two-blocks estimate]
\label{l3}
\begin{equation*}
\nop_\gamma \varlimsup_{\vphantom{0}\ell \to \infty} \varlimsup_{\epsilon \to 0} \varlimsup_{\vphantom{0}n \to \infty} \nop_g \big\{ \gamma \<\mc W_f^\ell - \mc W_f^{\epsilon n} , g^2\> + \frac{1}{n}\<g, \mc L_n\>\big\} <+\infty.
\end{equation*}
\end{lemma}
In order to prove this lemma, let us first define, for $\xi \in \Omega_n$, $x \in \bb T_n$ and $\ell \leq n$,
\begin{equation*}
\xi^\ell(x) = \frac{1}{\ell} \sum_{y \in \Lambda_\ell} \xi(x+y).
\end{equation*}
This notation will not enter in conflict with $\xi_t^n$, since we will only use it in this section, where no reference to the evolution is made. Notice that $\mc W_f^\ell-\mc W_f^{\epsilon n} = \bar{f}(\xi^\ell(0))-\bar{f}(\xi^{\epsilon n}(0))$. Since the function $f$ is local, the function $\bar{f}$ is a polynomial, and in particular it is uniformly Lipschitz on  $[0,1]$. Let $K_f$ be the corresponding Lipschitz constant. Let us assume that $\epsilon n$ is an integer multiple of $\ell$. The modifications needed if this is not the case will be evident. We have that
\begin{equation*}
\big|\mc W_f^\ell-\mc W_f^{\epsilon n}\big| \leq K_f \big|\xi^\ell(0) - \xi^{\epsilon n}(0)\big|
		\leq K_f \frac{\ell}{\epsilon n} \sum_y \big|\xi^\ell(0)-\xi^\ell(y)\big|,
\end{equation*}
where the sum is over sites $y \in \Lambda_{\epsilon n}$ which are multiple integers of $\frac{\ell}{n}$. The two blocks on the name of Lemma \ref{l3} are the two blocks of size $\ell$ on the right-hand side of this inequality. Using Proposition \ref{p11} and the inequality above, we see that to prove Lemma \ref{l3}, it is enough to show that 
\begin{equation}
\label{ec69}
\nop_\gamma \varlimsup_{\vphantom{0}\ell \to \infty} \varlimsup_{\epsilon \to 0} \varlimsup_{\vphantom{0} n\to \infty} \nop_y \nop_g \big\{ \gamma \<\big|\xi^\ell(0)-\xi^\ell(y)\big|,g^2\> - n \mc D(g)\big\} <+\infty.
\end{equation}
The proof of this inequality is very similar to the proof of the one-block estimate, therefore we will not give the full details in the derivation of those steps which are also present in the proof of Lemma \ref{l2}. Let $\mc F_\ell^y = \sigma\{\xi(x), \xi(x+y); x \in \Lambda_\ell\}$. We can restrict the supremum to non-negative functions $g$ with $\<g,g\>=1$. For a given non-negative $g$, define $g_{\ell,y}=E_{\nu_\rho}[g^2|\mc F_\ell^y]^{\smash{\frac{1}{2}}}$. Define
\begin{equation*}
\mc D_\ell^y(g) = \sum_{\mathclap{\substack{x,z \in \Lambda_\ell\\x \sim z}}}\big\{ \mc D^{x,z}(g) + \mc D^{x+y,z+y}(g)\big\}.
\end{equation*}
By Proposition \ref{p12}, $\mc D_\ell^y(g_{\ell,y}) \leq \mc D(g)$. The main difference between the one-block and two-block estimates is the following. The dynamics corresponding to the Dirichlet for $\mc D_\ell^y(\cdot)$ corresponds to two exclusion processes evolving in the two blocks separately. Therefore, a term connecting these two dynamics is needed. Let us define $\mc D_{\ell,\ast}^y(g) = \mc D_\ell^y(g)+\mc D^{\smash{\frac{1}{n},y+\frac{1}{n}}}(g)$. This new Dirichlet form connects what happens in the two boxes, through exchanges of particles between the first site of each box. The following {\em path lemma} tells us how to estimate $\mc D_{\ell,\ast}^y(g)$ in terms of $\mc D(g)$.

\begin{lemma}[Path lemma]\label{path}
For any $g:\Omega_n \to \bb R$ and any $y \in \bb T_n$, 
\begin{equation*}
\mc D^{\frac{1}{n},y+\frac{1}{n}}(g) \leq 4|y|n \mc D(g).
\end{equation*}
\end{lemma}
\begin{proof}[Proof of Lemma \ref{path}]
To simplify the notation, we switch to $\Omega =\{0,1\}^\bb Z$ and we consider $y=\ell-1$, $\ell \in \bb N$. For any permutation $\sigma: \Lambda_\ell \to \Lambda_\ell$ and any $\xi \in \Omega$, let us define $\xi^\sigma \in \Omega$ as
\begin{equation*}
\xi^\sigma(x)=
\begin{cases}
\xi(\sigma_x), &x \in \Lambda_\ell\\
\xi(x), &\text{ otherwise}.
\end{cases}
\end{equation*}
According to this notation,
\begin{equation*}
\mc D^{1,\ell}(g) = \frac{1}{2} \int \big(g(\xi^{(1\ell)})-g(\xi)\big)^2 \nu_\rho(d\xi).
\end{equation*}
Notice that $(1\;\ell) =(1\;2)\dots(\ell-1\;\ell)(\ell-2 \; \ell-1)\dots(1\;2)$, that is, the transposition $(1 \; \ell)$ is the product of $2\ell-3$ transpositions between neighbors. Let us denote by $\sigma_j$ the product of  the first $j$ transpositions. Since the measure $\nu_\rho$ is exchangeable, for any two permutations $\sigma$, $\tau$,
\begin{equation*}
\int \big(g(\xi^{\sigma \tau})-g(\xi^\tau)\big)^2 \nu_\rho(d\xi) = \int \big( g(\xi^\sigma)-g(\xi)\big)^2 \nu_\rho(d\xi).
\end{equation*}
Let us write $g(\xi^{(1\;\ell)})-g(\xi)$ as a telescopic sum:
\begin{equation*}
g(\xi^{(1\;\ell)})-g(\xi) = \sum_{j=1}^{2\ell-3} \big(g(\xi^{\sigma_j})-g(\xi^{\sigma_{j-1}})\big).\end{equation*}
By Cauchy-Schwarz inequality,
\begin{equation*}
\frac{1}{2} \int \big(g(\xi^{(1\;\ell)})-g(\xi)\big)^2 \nu_\rho(d\xi) \leq \frac{2\ell-3}{2} \sum_{j=1}^{2\ell-3} \int \big(g(\xi^{\sigma_j\sigma_{j-1}^{-1}})-g(\xi)\big)^2 \nu_\rho(d\xi).
\end{equation*}
Notice that $\sigma_j\sigma_{j-1}^{-1}$ is a transposition between neighbors, and notice as well that each pair of neighbours appears at most twice on the sum on the right-hand side of this inequality. Since $2(2\ell-3) \leq 4\ell$, the path lemma is proved.
\end{proof}

\begin{proof}[Proof of Lemma \ref{l3}]
Using Lemma \ref{path}, we see that $\mc D_{\ell,\ast}^y(g) \leq (1+4\epsilon n) \mc D(g)$ for any $g: \Omega_n \to \bb R$. Therefore, in order to show \eqref{ec69} and hence Lemma \ref{l3}, it is enough to show that for any $\gamma >0$,
\begin{equation*}
\varlimsup_{\vphantom{0}\ell \to \infty} \varlimsup_{\epsilon \to 0} \varlimsup_{\vphantom{0} n \to \infty} \nop_{y} \nop_{g} \big\{\gamma \<\big|\xi^\ell(0)-\xi^\ell(y)\big|,g^2\> - \frac{1}{\frac{1}{n}+4\epsilon}  \mc D_{\ell,\ast}^{y}(g)\big\}=0.
\end{equation*}

For the reader who knows the theory of hydrodynamic limits, this variational problem is essentially the same appearing in the middle of page 93 of \cite{KL}, and in particular, they may skip the rest of the proof. Let us identify the set of $\mc F_\ell^y$-measurable functions with the set of functions from $\Omega_{\ell}^2= \{0,1\}^{\Lambda_\ell} \times \{0,1\}^{\Lambda_\ell}$ to $\bb R$. Let us denote by $(\xi,\zeta)$ the elements of $\Omega_\ell^2$. With this identification, we can rewrite the supremum above as
\begin{equation*}
\nop_g\big\{\gamma \<\big|\xi^\ell(0)-\zeta^\ell(0)\big|, g^2\> - \frac{1}{\frac{1}{n}+4\epsilon} \mc D_{\ell,\ast}(g)\big\},
\end{equation*}
where the supremum is over non-negative functions $g: \Omega_\ell^2 \to \bb R$ such that $\<g,g\>=1$. Notice that the dependence on $y$ has been totally washed away. Repeating the compactness argument given in the proof of the one-block estimate, this time with $\epsilon$ playing the role of $n$, we are left to proving that 
\begin{equation*}
\varlimsup_{\ell \to \infty} \nop_{0\leq k\leq 2\ell} \int \big|\xi^\ell(0)-\zeta^\ell(0)\big| d\nu_{k,\ell}^2=0,
\end{equation*}
where $\nu_{k,\ell}^2$ is the uniform measure on the set
\begin{equation*}
\Big\{ (\xi,\zeta) \in \Omega_\ell^2; \sum_{x\in \Lambda_\ell} \big\{\xi(x)+\zeta(x) =k\big\}\Big\}.
\end{equation*}
It turns out that it is simpler to compute 
\begin{equation*}
\int \big(\xi^\ell(0)-\zeta^\ell(0)\big)^2 d\nu_{k,\ell}^2=0= \int \big(\xi^\ell(0)-\xi^\ell(\ell)\big)^2 d\nu_{k,2\ell}.
\end{equation*}
In fact, it is enough to observe that $\int \xi(x) d\nu_{k,2\ell}=\frac{k}{2\ell}$ and that $\int \xi(x)\xi(y) d\nu_{k,2\ell} = \frac{k(k-1)}{2\ell(2\ell-1)}$. With these two computations in hand, we can show that the variance above is equal to $\frac{k(2\ell-k)}{\ell^2(\ell-1)} \leq \frac{1}{2\ell-1}$, which finishes the proof of Lemma \ref{l3}.
\end{proof}

\subsection{Final remarks} 
In the previous four subsections, we have proved the local superexponential estimate, Lemma \ref{l1}. It turns out that in its current form, this is {\em not} what we need in order to deal with the martingales $\{\mc M_t^{a,n}; t \in [0,T]\}$. The problem is that the local function appearing there also depends on time. Recalling the bound in \eqref{ec53}, we see a constant $t$ multiplying the supremum on the right-hand side of the inequality. This constant can be changed into an integration over $[0,t]$, if the local function $f$ depends on $t$ as well. We did not include this dependence on $t$ from the beginning because it would have overcharged an already heavy notation. In the application we have in mind, the dependence on $t$ is rather simple. In fact, $f(\xi) = c^\pm(\xi)(e^{a(t)}-1)$. Therefore, the constant $e^{a(t)}-1$ could have been absorbed into $\gamma$ during all the computations, and in the end what we could prove is that the local superexponential estimate remains true 
whenever $a:[0,T] \to \bb R$ remains bounded. If the reader is not satisfied with this sketch, here is a different argument. Recall that in the construction of the martingale $\{\mc M_t^{a,n}; t \in [0,T]\}$ we are assuming that $a \in \mc C^1$. Actually for the argument we will explain, continuity is enough. Since $f$ is bounded, given $\delta >0$ it is possible to find $\delta'>0$ such that $|e^{a(t)}-e^{a(s)}| \leq \frac{\delta}{2T}$ if $|s-t|\leq \delta'$. Therefore, we can approximate $e^{a(t)}-1$ by a function which is piecewise constant on finite intervals of size at most $\delta'$, with an error at most $\frac{\delta}{2}$. On each one of these finite intervals we can use the local superexponential estimate, proving the extension to the time-dependent function $c^{\pm}(\xi) (e^{a(t)}-1)$. Since we will only need the superexponential estimate for these functions, we state it as a lemma:

\begin{lemma}
\label{l4}
For any $t \in [0,T]$, any $\delta>0$ and any continuous function $a: [0,T] \to \bb R$,
\begin{equation*}
\varlimsup_{\epsilon \to 0} \varlimsup_{\vphantom{0}n \to \infty} \frac{1}{n} \log \bb P_n \Big( \Big| \int_0^t \big\{c^{\pm}(\eta_s^n; x_s^n)-v^{\pm}(\pi_s^n(\iota_\epsilon(x_s^n)))\big\}\big(e^{\pm a(s)}-1\big)ds\Big|>\delta \Big) =-\infty.
\end{equation*}
\end{lemma}

\section{The energy estimate}
\label{s5}
Let $u:[0,T] \times \bb T \to \bb R$ be a function of class $\mc C^{0,1}$. The{ \em energy} of the function $u$ is defined as
\begin{equation*}
\int_0^T \int_{\bb T} \big(\nabla u(t,x)\big)^2 dx dt.
\end{equation*}
Recall that according to our definition of the empirical measure $\{\pi_t^n; t \in [0,T]\}$ in terms of finite elements, $\pi_t^n(x)$ has finite energy for any $n \in \bb N$. Our aim will be to show that in some sense, the probability of trajectories with very high energy is very small. Recall that $\{\pi_t^n; t \in [0,T]\}$ is a very oscillatory object at local scales, so a na\"ive approach does not work. Indeed, we will need a {\em variational characterization} of the energy. Therefore, let us introduce some Hilbert spaces. For $f:[0,T] \times \bb T \to \bb R$,\footnote{In this section we will only use $f$ for test functions; do not confuse with the notation local functions used in the previous section}
define
\begin{equation*}
\|f\|_{0,T} = \Big(\int_0^T \int_{\bb T} f(t,x)^2 dx dt\Big)^{\frac{1}{2}}.
\end{equation*}
Let us denote by $\mc H_{0,T}$ the Hilbert space $\{f; \|f\|_{0,T}<\infty\}$. For $f,g \in \mc H_{0,T}$, define
\begin{equation*}
\linn f,g\rinn_{0,T} = \int_0^T \!\!\!\!\int_{\bb T} f(t,x) g(t,x) dx dt.
\end{equation*}
For $f \in \mc H_{0,T}$, let us define
\begin{equation*}
\|f\|_{1,T} =\;\sup_{\mathclap{\substack{h \in \mc C^{0,1}\\ \|h\|_{0,T}=1}}} \;\linn f,\nabla h\rinn_{0,T}.
\end{equation*}
We denote by $\mc H_{1,T}$ the space of functions $f \in \mc H_{0,T}$ such that $\|f\|_{1,T}<\infty$. Notice that $\mc H_{1,T}$ is not a Hilbert space: functions that are constant in space and such that $\int_0^T f(t)^2dt <\infty$ belong to $\mc H_{1,T}$ and satisfy $\|f\|_{1,T}=0$. In fact, if we say that $f \sim g$ whenever $f-g =: \lambda$ does not depend on $x$, then $\mc H_{1,T}/\sim$ is a Hilbert space. We will not use this fact, but we will use the following:

\begin{proposition}\label{prop:H}
If $\|f\|_{1,T} <\infty$, then there exists a function $\nabla f \in \mc H_{0,T}$ such that $\|f\|_{1,T}= \|\nabla f \|_{0,T}$ and moreover $\linn f, \nabla h \rinn_{0,T} = - \linn \nabla f, h \rinn_{0,T}$ for any $h$ of class $\mc C^{0,1}$. In addition, the function $x \mapsto f(t,x)$ is continuous for a.e.~$t \in [0,T]$.
\end{proposition}
\begin{proof}
The existence of $\nabla f$ is guaranteed by Riesz's representation theorem. By Fubini's Theorem, $\int (\nabla f_t)^2 dx < +\infty$ for almost every $t \in [0,T]$. And by Sobolev's Embedding Theorem, $\int (\nabla f_t)^2 dx < +\infty$ implies that $f$ is H\"older-continuous of index $1/2$.
\end{proof}
Let $\{h^j;j \in \bb N\}$ be a sequence of functions in $\mc C^{0,1}$, dense in the unitary ball of $\mc H_{0,T}$. Then, we can restrict the supremum in the variational formula of $\|f\|_{1,T}$ to the set $\{h^j; j \in \bb N\}$:
\begin{equation*}
\|f\|_{1,T} = \nop_j \linn f,\nabla h^j\rinn_{0,T}.
\end{equation*}
Throughout this section, we will denote by $\pi^n_\cdot$ the process $\{\pi_t^n; t \in [0,T]\}$, and we will denote by $\pi^n$ (without the dot) the function 
\begin{equation*}
\eta \mapsto \sum_{x \in \bb T_n} \eta(x) \delta_x^n(y) dy
\end{equation*}
from $\Omega_n$ to $\mc M_{0,1}^+(\bb T)$.
\begin{lemma}[Energy estimate]
\label{l5}
There exists a constant $C_0 \in (0,\infty)$ such that for any $M>0$, and any $\ell \in \bb N$,
\begin{equation*}
\varlimsup_{n \to \infty} \frac{1}{n} \log \bb P_n \big( \nop_{1\leq j \leq \ell}\linn \pi^n_\cdot\!\!,\nabla h^j \rinn_{0,T} >M\big) \leq C_0 - \frac{M}{2}.
\end{equation*}
\end{lemma}
\begin{proof}[Proof of Lemma \ref{l5}]
By Lemma \ref{p10}, it is enough to show that
\begin{equation*}
\varlimsup_{n \to \infty} \frac{1}{n} \log \bb P_n \big( \linn \pi^n_\cdot\!\!,\nabla h^j \rinn_{0,T} >M\big) \leq C_0 - \frac{M}{2}
\end{equation*}
for any $j \in \bb N$. Using the exponential Chebyshev's inequality, we see that for any $\gamma>0$ and the computations done in Section \ref{s3.1},
\begin{equation*}
\begin{split}
\frac{1}{n} \log \bb P_n\big(\linn \pi^n_\cdot\!\!,\nabla h^j
	&\rinn_{0,T}>M\big) 
	\leq -\gamma M +\frac{1}{n} \log \bb E_n\big[ e^{\gamma n \linn \pi^n_{\cdot}\!\!,\nabla h^j\rinn_{0,T}}\big]\\
	&\leq -\gamma M+ K_0+K_1T + \int_0^T \nop_g\big\{\gamma \< \pi^n(\nabla h^j_t), g^2\> - n \mc D(g)\big\} dt.
\end{split}
\end{equation*}
Therefore, we need to estimate the supremum on the right-hand side of this equation. The way to estimate this term is different from what we did in Sections \ref{s3.3} and \ref{s3.4}. Using the definition of $\pi^n$, we see that
\begin{equation*}
\pi^n(\nabla h^j_t) = \sum_{x \in \bb T_n} \big(\eta(x)-\eta(x+\tfrac{1}{n})\big) \left( h_t^j(x) + r_n^j(t,x) \right),
\end{equation*}
with $r_n^j(t,x)$ a correction of order $1/n$.
Therefore,
\begin{equation*}
\< \pi^n(\nabla h^j_t),g^2\> = \sum_{x \in \bb T_n} \left( h_t^j(x) + r_n^j(t,x) \right) \<\eta(x)-\eta(x+\tfrac{1}{n}),g^2\>.
\end{equation*}
We will use the following trick: for any $x,y \in \bb T^n$ and any $g: \Omega_n \to \bb R$ such that $\<g,g\>=1$,
\begin{equation*}
\begin{split}
\<\eta(x) - \eta(y), g^2\> &= \<\eta(x), g^2(\eta)-g^2(\eta^{x,y})\>\\
		&=\big\<\eta(x), \big(g(\eta) -g(\eta^{x,y})\big)\big(g(\eta)+g(\eta^{x,y})\big)\big\>\\
		&\leq \frac{1}{2\beta_{x,y}} \mc D^{x,y}(g) + \frac{\beta_{x,y}}{2} \int \big(g(\eta) + g(\eta^{x,y})\big)^2 \nu_\rho(d\eta)\\
		&\leq \frac{1}{2\beta_{x,y}} \mc D^{x,y}(g) + 2\beta_{x,y}
\end{split}
\end{equation*}
for any $\beta_{x,y} >0$. In the first inequality we used the weighted Cauchy-Schwarz inequality and the inequality between arithmetic and geometric mean. In the second inequality we used the fact that $\<g,g\>=1$. Choosing $\beta_{\smash{x,x+\frac{1}{n}}} = \frac{\gamma |h_t^{\smash{j}}(x) + r_n^j(t,x)|}{2n}$, we obtain the bound
\begin{equation*}
\gamma \<\pi^n(\nabla h^j_t), g^2\> - n \mc D(g) \leq \frac{\gamma^2}{n} \sum_{x \in \bb T_n} \left( h_t^j(x) + r_n^j(t,x) \right)^2,
\end{equation*}
valid for any $g: \Omega_n \to \bb R$ with $\<g,g\>=1$. We conclude that
\begin{equation*}
\frac{1}{n} \log \bb P_n\big( \linn \pi^n_\cdot,h^j\rinn_{0,T}>M\big) \leq -\gamma M+K_0+K_1T+ \gamma^2 \! \int_0^T \!\frac{1}{n}\! \sum_{x \in \bb T_n} \!\!\!\!\left( h_t^j(x) + r_n^j(t,x) \right)^2 dt.
\end{equation*}
Therefore, sending $n$ to $\infty$ we see that for $C_0=K_0+K_1T$,
\begin{equation*}
\varlimsup_{n \to \infty} \frac{1}{n} \log \bb P_n \big( \linn \pi^n_\cdot\!\!,\nabla h^j \rinn_{0,T} >M\big) \leq -\gamma M +C_0 +\gamma^2.
\end{equation*}
Minimizing over $\gamma$ concludes the proof.
\end{proof}

\section{The upper bound}
\label{s6} 
Now that we have the superexponential estimate and the energy estimate at our disposal, we can show the large deviation upper bound on Theorem \ref{t2}. As we have done before, for the sake of  clarity, we break the proof into various steps.

\subsection{The upper bound for open sets}

Let us recall that we want to obtain a large deviation principle for the pair \newline $\{(\pi_t^n; x_t^n); t\in [0,T]\}$, viewed as a random variable with values in the Polish space 
$\mc E=\mc D([0,T]; \mc M_{0,1}^+(\bb T)) \times \mc M^+([0,T])\times \mc M^+([0,T])$. Recall that we are identifying the process $\{x_t^n; t \in [0,T]\}$ with the pair of positive Radon measures $(\omega_-^n, \omega_+^n)$, corresponding to the derivatives of the processes $\frac{1}{n} N_t^{n,-}$, $\frac{1}{n} N_t^{n,+}$.
\newline The space $\mc D([0,T]; \mc M_{0,1}^+(\bb T))$ is equipped with the $J_1$-Skorohod topology, while $\mc M^+([0,T])$ is equipped with the weak topology. 

Notation will become cumbersome very quickly, unless we adopt some simplifying conventions. We will denote the process $\{(\pi_t^n; x_t^n); t \in [0,T]\}$ by $(\pi^n\!\!,x^n)$. In particular, we abandon the notation introduced in Section \ref{s4}, where we used the notation $\pi^n_\cdot$ (with a dot) for $\{\pi_t^n; t \in [0,T]\}$.
Let $\mc A \subseteq \mc E$ be an open set and $\{\mc M_t^n; t \in [0,T]\}$ be a positive martingale with unit expectation. Assume that $\mc M_T^n$ is a function of $(\pi^n\!\!,x^n)$. Then,
\begin{equation*}
\begin{split}
\frac{1}{n} \log \bb P_n(\mc A) 
		&= \frac{1}{n} \log \bb E_n\big[ \mc M_T^n \big(\mc M_T^n\big)^{-1} \mathbf{1}_{\mc A}\big]\\
		&\leq \sup_{(\pi^n\!\!,x^n)\in \mc A}\frac{1}{n} \log \left( \mc M_T^n \right)^{-1}.\\
\end{split}
\end{equation*}
The martingales $\{\mc M_t^{a,n}; t \in [0,T]\}$, $\{\mc M_t^{H,n}; t \in [0,T]\}$ are not functions of $(\pi^n\!\!,x^n)$ but the superexponential estimates of Lemma \ref{p8} and Lemma \ref{l4} say that these martingales can be approximated by functions of $(\pi^n\!\!,x^n)$, with an error that is superexponentially small.
To keep track of all the indices and ease the reading, we now introduce some notation. Let us denote by $I$ the set of indices $i$ of the form $i=\{v_0,a,H,\epsilon,\delta,\ell,M\}$, where $v_0:\bb T \to [0,1]$ is continuous, $a:[0,T] \to \bb R$ is of class $\mc C^1$, $H:[0,T]\times \bb T \to \bb R$ is of class $\mc C^{1,2}$, $\epsilon>0$, $\delta>0$, and $\ell, M \in \bb N$. In what follows, we use the index $i$ to denote dependence on some (sometimes all, but not always) of the variables $\{v_0,a,H, \epsilon,\delta, \ell, M\}$. We start by preparing an initial distribution associated to a profile $v_0$. For $v_0: \bb T \to [0,1]$ continuous, define
\begin{equation*}
f(x) = \log\frac{v_0(x) \big(1-u_0(x)\big)}{u_0(x)\big(1-v_0(x)\big)}.
\end{equation*}
Recall the definition of $\{\rho_x^n; x \in \bb T_n\}$ given in Section \ref{s1.6}. Define the functions  $f_x^n = n \int \delta_x^n(y) f(y) dy$ and 
\begin{equation*}
v_x^n = \frac{\rho_x^n e^{f_x^n}}{1+\rho_x^n(e^{f_x^n}-1)}.
\end{equation*}
Define $\hat{\nu}_{v_0}^n$ as the product measure in $\Omega_n$ given by
\begin{equation*}
\hat{\nu}_{v_0}^n(\eta) = \prod_{x \in \bb T_n} \big\{ v_x^n \eta(x) + (1-v_x^n) (1-\eta(x))\big\}.
\end{equation*}
Notice that with this definition, the Radon-Nikodym derivative $\frac{d\hat{\nu}_{v_0}^n}{d\nu_{u_0}^n}$ is a function of the empirical density $\pi_0^n$. The process $\{(\eta_t^n; x_t^n); t \in [0,T]\}$ with initial distribution $\hat{\nu}_{v_0}^n$ has distribution $\frac{d\hat{\nu}_{v_0}^n}{d\nu_{u_0}^n} \bb P_n$.
Recall \eqref{Ma} and \eqref{MH}, and consider the martingale $\{\mc M_t^{i,n}; t \in [0,T]\}$ given by
\begin{equation}\label{eq:Mi}
\mc M_t^{i,n} = \frac{d\hat{\nu}_{v_0}^n}{d\nu_{u_0}^n} \mc M_t^{a,n} \mc M_t^{H,n}.
\end{equation}
Let $\mc U_i^n= \mc U_{\epsilon,\delta}^{H,n} \cap \mc U_{\epsilon, \delta}^{a,n} \cap \mc U_{M,\ell}^n$ denote the intersection of the sets
\begin{equation}\label{eq:U^H}
\mc U_{\epsilon, \delta}^{H,n} = 
\Big\{\Big|\mc Q_T^n(H) - \int_0^T \frac{1}{n} \sum_{x \in \bb T_n}  \big(\nabla H_t(x)\big)^2 \pi_t^n(\iota_\epsilon(x))\big(1-\pi_t^n(\iota_\epsilon(x))\big) dt \Big| \leq \delta\Big\},
\end{equation}
\begin{equation}\label{eq:U^a}
\mc U_{\epsilon,\delta}^{a,n}=
\Big\{\Big| \int_0^t \big\{c^{\pm}(\eta_s^n; x_s^n)-v^{\pm}(\pi_s^n(\iota_\epsilon(x_s^n)))\big\}\big(e^{\pm a(s)}-1\big)ds\Big|\leq \delta \Big\},
\end{equation}
\begin{equation*}
\mc U_{M,\ell}^n = \Big\{ \nop_{1\leq j \leq \ell} \linn \pi^n,h^j\rinn_{0,T} \leq M \Big\}
\end{equation*}
Lemma \ref{l4} and Lemma \ref{l5} imply that
\begin{equation*}
\varlimsup_{\epsilon \to 0} \varlimsup_{\vphantom{0}n \to \infty} \frac{1}{n} \log \bb P_n \big(\big(\mc U_{\epsilon,\delta}^{H,n}\big)^{c}\big) = -\infty,
\end{equation*}
\begin{equation*}
\varlimsup_{\epsilon \to 0} \varlimsup_{\vphantom{0} n \to \infty} \frac{1}{n} \log \bb P_n \big(\big(\mc U_{\epsilon,\delta}^{a,n}\big)^{c}\big) = -\infty,
\end{equation*}
\begin{equation*}
\varlimsup_{n \to \infty} \frac{1}{n} \log \bb P_n \big(\big(\mc U_{M,\ell}^{n}\big)^{c}\big) \leq C_0-\tfrac{M}{2}.
\end{equation*}
Therefore,
\begin{equation*}
\varlimsup_{n \to \infty} \frac{1}{n} \log \bb P_n \big(\big(\mc U_i^{n}\big)^{c}\big) \leq \max\big\{U_{\epsilon,\delta}^{a,H}, C_0-\tfrac{M}{2}\big\},
\end{equation*}
where $U_{\epsilon, \delta}^{a,H}$ is a constant which converges to $-\infty$ as $\epsilon \to 0$, regardless of the values of $\delta, a$ or $H$. Therefore
\begin{equation*}
\begin{split}
\varlimsup_{n \to \infty} \frac{1}{n} \log \bb P_n((\pi^n\!\!,x^n) 
		&\in \mc A) \leq\\
		&\leq \varlimsup_{n \to \infty} \frac{1}{n} \log 2\max\Big\{  \bb P_n(\{(\pi^n\!\!,x^n) \in \mc A\} \cap \mc U_i^n), \bb P_n\big(\big(\mc U_i^n\big)^{c}\big)\Big\}\\
		& \leq \max\Big\{ \varlimsup_{n \to \infty} \frac{1}{n}\log \bb P_n(\{(\pi^n\!\!,x^n) \in \mc A\} \cap \mc U_i^n), U_{\epsilon,\delta}^{a,H}, C_0-\tfrac{M}{2}\Big\}.\\
\end{split} 
\end{equation*}
On the set $\mc U_i^n$, the martingale $\mc M_T^{i,n}$ is a function of the pair $(\pi^n\!\!,x^n)$, plus some small error term. Consequenlty, we can bound
\begin{equation*}
\begin{split}
\frac{1}{n} \log \bb P_n(\{(\pi^n\!\!,x^n) \in \mc A \} \cap \mc U_i^n) \leq \sup_{(\pi\!,x) \in\mc A \cap \mc U_M^\ell} \big\{ - 
		&\big( j_\epsilon(a;\pi,x) + J_\epsilon^n(H; \pi)+\\
		&+h_n(v_0,u_0;\pi_0) \big)+r_n(H) +2\delta\big\},
\end{split}
\end{equation*}
where $\mc U_M^\ell = \{\nop_{1\leq j\leq \ell} \linn \pi,\nabla h^j\rinn_{0,T} \leq M\}$ and the functions $j_\epsilon$, $J_\epsilon^n$ and $h_n$ are given by
\begin{equation}\label{je}
j_\epsilon(a; \pi,x) = a(T) x_T - \int_0^T\Big\{ a'(t)x_t + \sum_{z=\pm} v^z(\pi_t(\iota_\epsilon(x_t)))(e^{za(t)}-1)ds\Big\},
\end{equation}
\begin{equation*}
\begin{split}
 J_\epsilon^n(H;\pi) = \pi_T( H_T) 
 		&- \pi_0( H_0) - \int_0^T \pi_t \big(\partial_t H_t + 2\Delta H_t\big) d \pi_t dt\\
		&\quad -\int_0^T \frac{1}{n} \sum_{x \in \bb T_n} \big( \nabla H_t(x)\big)^2 \pi_t(\iota_\epsilon(x))\big(1-\pi_t(\iota_\epsilon(x))\big)dt,
 \end{split}
\end{equation*}
\begin{equation*}
h_n(v_0,u_0; \pi_0) = \int \log \frac{v_0(x)(1-u_0(x))}{u_0(x)(1-v_0(x))} d \pi_0 + \frac{1}{n} \sum_{x \in \bb T_n} \log\frac{1-v_x^n}{1-\rho_x^n}.
\end{equation*}
Recall that the error term $r_n(H)$ comes from replacing a discrete version of the Laplacian of $H$ by $\Delta H$. The error term $2 \delta$ comes from the use of the superexponential estimates stated in Lemma \ref{p8} and Lemma \ref{l4}. 
Using the smoothness of $\nabla H_t$ and of $v_0$, we see that $h_n$ and $J_\epsilon^n$ converge to the functions
\begin{equation}\label{Je}
\begin{split}
 J_\epsilon(H;\pi) = \pi_T( H_T) 
 		&- \pi_0( H_0) - \int_0^T \pi_t \big(\partial_t H_t + 2\Delta H_t\big) d \pi_t dt\\
		&\quad -\int_0^T \int  \big( \nabla H_t(x)\big)^2 \pi_t(\iota_\epsilon(x))\big(1-\pi_t(\iota_\epsilon(x))\big)dx dt,
\end{split}
\end{equation}
\begin{equation*}
h(v_0,u_0; \pi_0) = \int \log \frac{v_0(x)(1-u_0(x))}{u_0(x)(1-v_0(x))} d \pi_0 + \int \log\frac{1-v_0(x)}{1-u_0(x)}dx.
\end{equation*}
Let us define
\begin{equation}\label{Ji}
\mc J_i(\pi,x) =
\begin{cases}
j_\epsilon(a; \pi,x) + J_\epsilon(H;\pi) + h(v_0,u_0; \pi_0) - 2\delta, &\text{if } (\pi,x) \in \mc U_M^\ell\\
+\infty, &\text{otherwise}.
\end{cases}
\end{equation}
The function $\mc J_{i}(\pi,x)$ is lower semicontinuous, since each one of the functions $j_\epsilon$, $\mc J_\epsilon$ and $h$ are continuous, and the set $\mc U_M^\ell$ is closed. Minimizing over all the indices $i$, we finally obtain the upper bound for open sets:
\begin{equation}\label{Aub}
\varlimsup_{n \to \infty} \frac{1}{n} \log \bb P_n((\pi^n\!\!,x_n) \in \mc A) \leq \inf_{i \in I\vphantom{\mc A}} \nop_{(\pi,x) \in \mc A} \max\{-\mc J_i(\pi,x), U_{\epsilon,\delta}^{a,H}, C_0-\tfrac{M}{2}\}.
\end{equation}

\subsection{The upper bound for compact sets}

Once a large deviation upper bound has been obtained for open sets, the standard way to pass from it to an upper bound for compact sets is through the so-called {\em Minimax lemma}, whose proof can be found in \cite{KL}, Lemma 3.2 in Appendix 2.

\begin{proposition}[Minimax Lemma] 
Let $\{\mc F_i; i \in I\}$ be a family of upper semicontinuous functions defined on a Polish space $\mc E$. Let $\{P_n; n \in \bb N\}$ be a sequence of probability measures in $\mc E$. Assume that for any open set $\mc A \subseteq \mc E$,
\begin{equation*}
\varlimsup_{n \to \infty} \frac{1}{n} \log P_n (\mc A) \leq \inf_{i \in I\vphantom{\mc A}} \nop_{x \in \mc A} \mc F_i(x).
\end{equation*}
Then, for any compact set $\mc K \subseteq \mc E$,
\begin{equation*}
\varlimsup_{n \to \infty} \frac{1}{n} \log P_n( \mc K) \leq \nop_{x \in \mc K} \inf_{i \in I\vphantom{\mc K}} \mc F_i(x).
\end{equation*}
\end{proposition}

Let $\mc K \subseteq \mc E$ be a compact set.
Applying the Minimax Lemma to the family of functions $\max\{-\mc J_i(\pi,x),U_{\epsilon,\delta}^{a,H}, C_0-\frac{M}{2}\}$ in \eqref{Aub}, we obtain the bound
\begin{equation}\label{Kub}
\varlimsup_{n \to \infty} \frac{1}{n} \log \bb P_n ( (\pi^n\!\!,x^n) \in \mc K) 
		\leq \nop_{(\pi,x) \in \mc K} \inf_{i \in I \vphantom{\mc K}}\;\; \max\;\{-\mc J_i(\pi,x), U_{\epsilon,\delta}^{a,H}, C_0-\tfrac{M}{2}\}.
\end{equation}
Recall that the index $i$ includes all the possible choices of $v_0$, $a$, $H$, $\epsilon$, $\delta$, $\ell$ and $M$. We will take advantage of this by taking the infima in the right order. Observe that we can replace $\inf$ by $\liminf$ whenever it is convenient, since the $\liminf$ of a sequence is greater than the $\inf$ of the same sequence. 
Recall the definition $\mc U_M^\ell = \{\nop_{1\leq j\leq \ell} \linn \pi,\nabla h^j\rinn_{0,T} \leq M\}$. Now we send $\ell \to \infty$. Notice that $\mc J_i(\pi,x)$ is increasing in $\ell$: the set where we define $\mc J_i(\pi,x)$ as equal to $+\infty$ is growing with $\ell$, and outside of it, the function $\mc J_i(\pi,x)$ does not depend on $\ell$. This is equivalent to saying that we are restricting the supremum to the intersection of $\mc K$ and $\mc U_M= \cap_\ell \mc U_M^\ell$. 
By the definition of the sequence $\{h^j; j \in \bb N\}$, the set $\mc U_M$ is equal to the set $\{\|\pi\|_{1,T}\leq M\}$. Now it is the turn of sending $M \to \infty$. Doing this, there are two effects. First, the term $C_0-\frac{M}{2}$ goes to $-\infty$, and we can take it out of the maximum. And second, the set $\mc U = \cup_M \mc U_M$ is equal to the set $\mc H_{1,T} = \{\|\pi\|_{1,T} <+\infty\}$. Therefore, after taking the limit in $\ell$ first and then in $M$, in view of \eqref{Kub}, we end up with the inequality:
\begin{equation}\label{KHub}
\varlimsup_{n \to \infty} \frac{1}{n} \log \bb P_n((\pi^n\!\!,x^n) \in \mc K) \leq \sup_{ \mc K \cap \mc H_{1,T}}\!\! \inf_{\vphantom{{\mc H}^{H}} i} \;\max\{-\mc J_i(\pi,x) , U_{\epsilon,\delta}^{a,H}\}.
\end{equation}
Notice that these two limit procedures together with Section \ref{s4} were devoted to maximize over the set $\mc K \cap \mc H_{1,T}$ instead of $\mc K$. The reason for this will become transparent in what follows. We now move to minimize the r.h.s. of \eqref{KHub} over $\epsilon$. Recall that $U_{\epsilon,\delta}^{a,H}$ goes to $-\infty$ as $\epsilon \to 0$ if the other parameters are fixed. But then we need to analyze the limit of $\mc J_i(\pi,x)$ in \eqref{Ji} when $\epsilon\to 0$, The analysis of the term $J_\epsilon(H; \pi)$ in \eqref{Je} has been already done in \cite{KOV}  and in Chapter 10 of \cite{KL}. Therefore, we just need to look at $j_\epsilon(a;\pi,x)$ in \eqref{je}.
When $\epsilon \to 0$, we cannot guarantee that $\pi_t(\iota_\epsilon(x_t))$ in \eqref{je} goes to $\pi_t(x_t)$ if we only know that $\pi_t$ has bounded density: it may easily be the case that $x_t$ is a non-removable-by-smoothing discontinuity for every $t \in [0,T]$. The set of points of this type forms a very thin subset of $\bb T$, but we cannot rule out  a pathological behavior supposing only that $\pi_t \in \mc M_{0,1}^+(\bb T)$. Since we can assume that $\pi \in \mc H_{1,T}$, we can also assume that $x \mapsto \pi_t(x)$ is continuous for a.e.~$t \in [0,T]$. Then, $\pi_t(\iota_\epsilon(x_t))$ converges to $\pi_t(x_t)$ for a.e.~$t \in [0,T]$. By the dominated convergence theorem, we conclude that $j_\epsilon(a;\pi,x)$ converges, as $\epsilon \to 0$, to
\begin{equation}\label{RateFn4}
j(a;\pi,x) = a(T) x_T - \int_0^T \Big\{a'(t) x_t + \sum_{z=\pm} v^z(\pi_t(x_t))(e^{za(t)}-1) \Big\}dt
\end{equation}
As shown in \cite{KOV}, as $\epsilon \to 0$, the function $J_\epsilon(H;\pi)$ also has a well-defined limit, and the fact that $\pi_t \in \mc M_{0,1}^+(\bb T)$ is enough to justify the limit. This limit is equal to
\begin{equation*}
\begin{split}
\lim_{\epsilon \to 0} J_\epsilon(H;\pi)=J(H;\pi) := \pi_T( H_T) 
 		&- \pi_0( H_0) - \int_0^T \pi_t \big(\partial_t H_t + 2\Delta H_t\big) d \pi_t dt\\
		&\quad -\int_0^T \int  \big( \nabla H_t(x)\big)^2 \pi_t(x)\big(1-\pi_t(x)\big)dx dt,
\end{split}
\end{equation*}
Finally, by taking $\epsilon \to 0$ in the r.h.s. of \eqref{KHub}, we get the bound
\begin{equation*}
\begin{split}
\varlimsup_{n \to \infty} \frac{1}{n} \log \bb P_n((\pi^n\!\!,x^n) \in \mc K)
		&\leq \sup_{ \mc K \cap \mc H_{1,T}} \inf_{\vphantom{{\mc H}^{h}} v_0,a,H}\; \{-(j(a;\pi,x) +J(H;\pi)+h(v_0,u_0;\pi_0))\}\\
		&= -\inf_{ \mc K \cap \mc H_{1,T}} \nop_{\vphantom{{\mc H}} v_0,a,H}\; \big\{j(a;\pi,x) +J(H;\pi)+h(v_0,u_0;\pi_0)\big\}.\\
\end{split}
\end{equation*}
It turns out that the last supremum is exactly the rate function of the large deviation principle stated in Theorem \ref{t2} (see equations (1.1)-(1.4) in Chapter 10 of \cite{KL} for the equivalence), and therefore we have completed the large deviation upper bound of Theorem \ref{t2} for compact sets. We state this bound as a lemma for further reference.

\begin{lemma}
\label{l6}
For any compact set $\mc K \subseteq \mc E$, 
\begin{equation*}
\varlimsup_{n \to \infty} \frac{1}{n} \log \bb P_n((\pi^n\!\!,x^n) \in \mc K) \leq -\!\!\!\!\inf_{(\pi,x) \in \mc K} \!\!\{\mc I_{\mathrm{rw}}(x|\pi)+ \mc I_{\mathrm{ex}}(\pi)\big\}
\end{equation*}
\end{lemma}

\subsection{Upper bound for closed sets}
The canonical way to extend a large deviation upper bound from compact sets to closed sets is to proving the {\em exponential tightness} of the corresponding sequence of processes. We say that the sequence $(\pi^n,x^n)$ is exponentially tight if for any $M >0$ there exists a compact $\mc K_M \subseteq \mc E$ such that 
\begin{equation*}
\varlimsup_{n \to \infty} \frac{1}{n} \log \bb P_n((\pi^n\!\!,x^n) \in \mc K_M^c) \leq -M.
\end{equation*}
The relevance of this condition is given by the following proposition, whose proof is left to the reader.

\begin{proposition}\label{expotight}
Let $\{P_n; n \in \bb N\}$ a sequence of probability measures defined on a Polish space $\mc E$. Let $\mc I: \mc E \to [0,\infty]$ be a lower semicontinuous function. Assume that for any compact set $\mc K \subseteq \mc E$,
\begin{equation*}
\varlimsup_{n \to \infty} \frac{1}{n} \log P_n(\mc K) \leq - \inf_{x \in \mc K} \mc I(x).
\end{equation*}
Assume in addition that the sequence $\{P_n; n \in \bb N\}$ is exponentially tight. Then,
\begin{equation*}
\varlimsup_{n \to \infty} \frac{1}{n} \log P_n(\mc C) \leq - \inf_{x \in \mc C} \mc I(x).
\end{equation*}
for any closed set $\mc C \subseteq \mc E$.
\end{proposition}

Due to the product structure of the state space of $(\pi^n\!\!,x^n)$, it is enough to show exponential tightness for each of the process $\{\pi^n; n \in \bb N\}$, $\{x^n; n \in \bb N\}$ separately. The exponential tightness of $\{\pi^n; n \in \bb N\}$ is proved in Chapter 10.4 of \cite{KL}, starting from eq.~(4.5). We are left to proving the exponential tightness of $\{x^n; n \in \bb N\}$. This is equivalent to showing the exponential tightness of each one of the processes $\{\omega_{\pm}^n; n \in \bb N\}$. Recall the following characterisation of compact sets of $\mc M^+([0,T])$. A closet set $\mc C \subseteq \mc M^+([0,T])$ is compact if and only if $
\sup_{\mu \in \mc C} \mu([0,T]) <+\infty.$
Notice as well that $\omega_\pm^n([0,T]) = \frac{1}{n} N_T^{\pm,n}$.
Therefore, in order to show exponential tightness of $\{\omega_{\pm}^n; n \in \bb N\}$, it is enough to show that
\begin{equation*}
\varlimsup_{M \to \infty} \varlimsup_{n \to \infty \vphantom{M}} \frac{1}{n} \log \bb P_n( N_T^{\pm,n} > nM) =-\infty.
\end{equation*}
This is actually simple to prove. In fact, the processes $\{M_t^{\pm,n}; t \in [0,T]\}$
\begin{equation*}
M_t^{\pm,n} = \exp\Big\{ \theta N_t^{\pm,n} -n\int_0^t c^\pm(\xi_s^n)(e^\theta-1) ds\Big\}
\end{equation*}
are positive martingales of unit expectation. In particular, taking $C_1 = \sup_\xi c^\pm(\xi)$,
\begin{equation*}
\bb E_n\big[e^{\theta N_t^{\pm,n}}\big] \leq e^{C_1 nt(e^\theta-1)}. 
\end{equation*}
Using the exponential Chebyshev's inequality, we see that
\begin{equation*}
\frac{1}{n} \log \bb P_n(N_T^{\pm,n} > nM) \leq C_1T(e^\theta-1)-\theta M,
\end{equation*}
which proves the exponential tightness of $\{x^n; n \in \bb N\}$. Therefore by Proposition \ref{expotight} and Lemma \ref{l6}, we conclude that
\begin{equation}\label{upperBound}
\varlimsup_{n \to \infty} \frac{1}{n} \log \bb P_n((\pi^n\!\!,x^n) \in \mc C) \leq -\!\!\inf_{(\pi,x) \in \mc C} \big\{\mc I_{\mathrm{rw}}(x|\pi) + \mc I_{\mathrm{ex}}(\pi)\big\},
\end{equation}
for any closed set $\mc C \subseteq \mc E$.

\subsection{Some properties of the rate function}\label{RateFn}
It turns out that a more explicit formula for the rate function $\mc I_{\mathrm{rw}}(x|\pi)$ can be obtained. Recall that we are assuming that $x$ has finite variation. We claim that $\mc I_{\mathrm{rw}}(x|\pi)=+\infty$ if $x$ is not absolutely continuous. Since $x$ has finite variation, we can justify an integration by parts to show that
\begin{equation*}
a(T) x_T-\int_0^T a'(t) x_t dt = \int_0^T a(t)dx_t.
\end{equation*}
Therefore,
\begin{equation*}
\mc I_{\mathrm{rw}}(x|\pi) = \nop_{a \in \mc C^1}\Big\{ \int_0^T a(t) dx_t - \int_0^T \sum_{z = \pm} v^z(\pi_t(x_t))(e^{za(t)}-1)dt\Big\}.
\end{equation*}
Let us assume that $x$ is not absolutely continuous. Then there exists a compact set $K \subseteq [0,T]$ such that $\int_0^T \mathbf{1}_K dx_t \neq 0$ and $\int_0^T \mathbf{1}_K dt =0$. For simplicity, we assume that $x(K)=:\int_0^T \mathbf{1}_K dx_t > 0$. Since $K$ is compact, there exists a sequence of smooth functions $a_\epsilon: [0,T] \to [0,1]$ such that $a_\epsilon \downarrow \mathbf{1}_K$ as $\epsilon \to 0$. Then, by the dominated convergence theorem,
\begin{equation*}
\lim_{\epsilon \to 0} \int_0^T \lambda a_\epsilon(t) d x_t = \lambda x(K),
\end{equation*}
\begin{equation*}
\lim_{\epsilon \to 0} \int_0^T \sum_{z=\pm} v^z(\pi_t(x_t))(e^{z\lambda a_\epsilon(t)}-1) dt =0.
\end{equation*}
Sending $\lambda \to \infty$, we conclude that $\mc I_{\mathrm{rw}}(x|\pi)=+\infty$. In particular, we can rewrite the rate function $\mc I_{\mathrm{rw}}$ as
\begin{equation}\label{eq:rw-rate}
\mc I_{\mathrm{rw}}(x|\pi) = \nop_{a \in \mc C^1} \int_0^T \Big\{ a(t) x'_t -\sum_{z=\pm} v^z(\pi_t(x_t))(e^{za(t)}-1)\Big\} dt.
\end{equation}
By an approximation argument, we can check that the supremum over $\mc C^1$ functions can be replaced by a supremum over bounded functions. An upper bound for $\mc I_{\mathrm{rw}}$ can be obtained by exchanging the supremum and the integration. The maximizing function $a$ is sensitive to $v^+(\pi_t(x_t))v^-(\pi_t(x_t))=0$. If $v^+(\pi_t(x_t))v^-(\pi_t(x_t))>0$, it is given by
\begin{align}\label{eq:explicitrate1}
\hat{a}_{x,\pi}(t) &= \log \frac{x'_t+\sqrt{(x'_t)^2+4v^+(\pi_t(x_t))v^-(\pi_t(x_t))}}{2v^+(\pi_t(x_t))}\\
&=- \log \frac{-x'_t+\sqrt{(x'_t)^2+4v^+(\pi_t(x_t))v^-(\pi_t(x_t))}}{2v^-(\pi_t(x_t))} .
\end{align} 
In general, the pointwise supremum of $a(t) x'_t -\sum_{z=\pm} v^z(\pi_t(x_t))(e^{za(t)}-1)$ is obtained at
\begin{align}\label{eq:explicitrate2}
a_{x,\pi}(t) &= 
\begin{cases}
\hat{a}_{x,\pi}(t), &v^+(\pi_t(x_t))v^-(\pi_t(x_t))>0\\
\log\frac{|x_t'|}{v^+(\pi_t(x_t))},\quad&v^+(\pi_t(x_t))v^-(\pi_t(x_t))=0,\ x_t'>0 \\
-\log\frac{|x_t'|}{v^-(\pi_t(x_t))},\quad&v^+(\pi_t(x_t))v^-(\pi_t(x_t))=0,\ x_t'<0 \\
-\infty,\quad&v^+(\pi_t(x_t))v^-(\pi_t(x_t))=0,\ x_t'=0,\ v^+(\pi_t(x_t))>0 \\
\infty,\quad&v^+(\pi_t(x_t))v^-(\pi_t(x_t))=0,\ x_t'=0,\ v^+(\pi_t(x_t))=0
\end{cases}
\end{align}
where we use the convention that $\infty\cdot0=0$.

If $a_{x,\pi}$ is bounded we have an explicit form for $\mc I_{\mathrm{rw}}$ by \eqref{eq:rw-rate}.

We show now that \eqref{eq:explicitrate2} is in fact always the optimizer. For simpler notation, we write $v^\pm_t=v^\pm(\pi_t(x_t))$. 
In a first step, we look at the finiteness of the rate function:
\begin{lemma}\label{lem:finite-rate}
The rate function $\mc I_{\mathrm{rw}}(x|\pi)$ is finite if and only if $x$ as absolutely continuous and
\begin{align}
\label{eq:finite1}&\int_0^T |x'_t|\log^+|x_t'| \,dt <\infty ,\\ 
\label{eq:finite2}&\int_0^T (x'_t)^+\log^+\left((v^+_t)^{-1}\right) \,dt <\infty \quad\text{and}\\
\label{eq:finite3}&\int_0^T (x'_t)^-\log^+\left((v^-_t)^{-1}\right) \,dt <\infty,
\end{align}
where $f^+=\max(f,0)$ and $f^-=\max(-f,0)$ are the positive and negative part of a function.
\end{lemma}
\begin{proof}
Finiteness of the rate function follows from 
\begin{align}\label{eq:int-finite2} 
\int_0^T \Big\{ a_{x,\pi}(t) x'_t -\sum_{z=\pm} v^z(\pi_t(x_t))(e^{za_{x,\pi}(t)}-1)\Big\} dt<\infty, 
\end{align}
which we now show under \eqref{eq:finite1}, \eqref{eq:finite2} and \eqref{eq:finite3}. First we observe that 
\begin{align}\label{eq:estimate1} 
0\leq \sum_{z=\pm} v^z_te^{za_{x,\pi}(t)}\leq 2|x_t'|+2\sqrt{v^+_tv^-_t}. 
\end{align}
Hence
\begin{align}\label{eq:int-finite1} 
\int_0^T \sum_{z=\pm}\left| v^z_t(e^{za_{x,\pi}(t)}-1)\right| dt<\infty 
\end{align}
by the absolute continuity of $x$ and the fact that $v^+$ and $v^-$ are bounded from above. 
Since the integrand in \eqref{eq:int-finite2} is non-negative, it follows from \eqref{eq:estimate1} and \eqref{eq:int-finite1} that we only need to look at the integrability of the positive part 
\[ (a_{x,\pi}(t)x_t')^+=a^+_{x,\pi}(t)(x_t')^++a_{x,\pi}^-(t)(x_t')^-. \] 
W.l.o.g. we look at $x'_t>0$. We have
\begin{align*}
\int\limits_{\mathrlap{\{t\in[0,T]:x_t'>0\}}} x'_t a_{x,\pi}^+(t)  \,dt& \leq \int\limits_{{\{t\in[0,T]:x_t'>0\}}} x'_t\left|\log\left(\frac{x_t'+\sqrt{x_t'+4v^+_tv^-_t}}{2v^+_t}\right)\right|  \,dt \\
&\leq \quad \int\limits_{\mathllap{\{t\in[0,T]:x_t'>0\}}}x'_t \left|\log\frac{x'_t+\sqrt{(x'_t)^2+4v^+_t v^-_t}}{2}\right|\,dt 
\ +\ \int\limits_{\mathclap{\{t\in[0,T]:x_t'>0\}}}x'_t \left|\log\left((v^+_t)^{-1}\right)\right|\,dt ,
\end{align*}
which is finite by \eqref{eq:finite1} and \eqref{eq:finite2}. 

For the other direction, assume that \eqref{eq:finite1}, \eqref{eq:finite2} or \eqref{eq:finite3} is infinite, which is equivalent to $\int_{\{t\in[0,T]:x_t'>0\}} |x'_t|\log^+\frac{|x'_t|}{v^+_t} \,dt =\infty$ or $\int_{\{t\in[0,T]:x_t'<0\}} |x'_t|\log^+\frac{|x'_t|}{v^-_t} \,dt =\infty$. To see that notice that since $v^\pm$ is bounded from above there is no relevant difference between $\log$ and $\log^+$, the integrals can only diverge if the argument of the logarithm diverges. So assume w.l.o.g. $\int_{\{t\in[0,T]:x_t'>0\}} |x'_t|\log^+\frac{|x'_t|}{v^+_t} \,dt =\infty$. Define for $K>0$ the bounded function
\begin{align*}
a_K(t):=\min\left(\log^+\left(\frac{x_t'+\sqrt{(x'_t)^2+4v^+_t v^-_t}}{2v^+_t}\right),K\right){\bb 1}_{x_t'>0}.
\end{align*}
Then, by \eqref{eq:rw-rate} and using the fact that $-v^-_t(e^{-a_K(t)}-1)\geq0$ and $-v^+_t(e^{a_K(t)}-1)\geq- \frac12\left( x_t'+\sqrt{(x'_t)^2+4v^+_t v^-_t}\right)$, we have
\begin{align*}
\mc I_{rw}(x|\pi)&\geq \int_0^T a_K(t) x'_t -\sum_{z=\pm}v^z_t\left(e^{z a_K(t)}-1\right)\,dt\\
&\geq \int_0^T a_K(t)x_t'- \frac12\left( x_t'+\sqrt{(x'_t)^2+4v^+_t v^-_t}\right)  \,dt.
\end{align*}
Since $x$ is absolutely continuous, there is some constant $M>0$ independent of $K$ so that
\begin{align*}
\mc I_{rw}(x|\pi)&\geq \int\limits_{\{t\in[0,T]:x_t'>0\}} x_t' \min\left(\log^+\left(\frac{x_t'+\sqrt{(x'_t)^2+4v^+_t v^-_t}}{2v^+_t}\right),K\right)\,dt -M \\
&\geq \int\limits_{\{t\in[0,T]:x_t'>0\}} x_t' \min\left(\log^+\left(\frac{x_t'}{v^+_t}\right),K\right)\,dt -M ,
\end{align*}
which by assumption diverges as $K\to\infty$.
\end{proof}

\begin{lemma}
The rate function $\mc I_{rw}(x|\pi)$ is given by
\begin{align*}
\mc I_{rw}(x|\pi)=\int_0^T \Big\{ a_{x,\pi}(t) x'_t -\sum_{z=\pm} v^z(\pi_t(x_t))(e^{za_{x,\pi}(t)}-1)\Big\} dt.
\end{align*}
\end{lemma}
\begin{proof}
Most of the work has been done in Lemma \ref{lem:finite-rate}, in particular we know that
\begin{align*}
\mc I_{rw}(x|\pi)\leq\int_0^T \Big\{ a_{x,\pi}(t) x'_t -\sum_{z=\pm} v^z(\pi_t(x_t))(e^{za_{x,\pi}(t)}-1)\Big\} dt
\end{align*}
and the left hand side is finite iff the right hand side is. So we only need to show that the right hand side is also a lower bound.
We define for $K>0$
\begin{align*}
a_K(t):=\max\left(\min\left(a_{x,\pi}(t),K\right),-K\right).
\end{align*}
Then
\begin{align*}
\mc I_{rw}(x|\pi)&\geq \limsup_{K\to\infty} \int_0^T a_K(t) x'_t -\sum_{z=\pm}v^z_t\left(e^{z a_K(t)}-1\right)\,dt.
\end{align*}
To move the limit inside the integral we will use dominated convergence. We claim that there is an $M>0$ so that
\begin{align*}
\left|a_K(t) x'_t -\sum_{z=\pm}v^z_t\left(e^{z a_K(t)}-1\right)\right| \leq |a_{x,\pi}(t)x_t'| + 4|x'_t| + 2M,
\end{align*}
which is integrable by Lemma \ref{lem:finite-rate}. The first term is clear, so we only look at the second term:
\begin{align*}
\left|v^z_t\left(e^{z a_K(t)}-1\right)\right| &\leq \left(|x_t'|+\sqrt{(x'_t)^2+4v^+_t v^-_t}+v_t^z\right){\bb 1}_{|a_K(t)|<K} \\
&\quad +\left(|x_t'|+\sqrt{(x'_t)^2+4v^+_t v^-_t}+v_t^z \right){\bb 1}_{a_K(t)=zK} \\
&\quad +v_t^z(1-e^{-K}){\bb 1}_{a_K(t)=-zK} \\
&\leq 2|x_t'|+M
\end{align*}
for some $M>0$ depending only on the upper bounds of $v_t^\pm$.
\end{proof}

\section{The lower bound}
\label{s7}
\subsection{Hydrodynamic limit for the perturbed system}
From now on we denote by $i$ a given choice of the triple $i=\{v_0,H,a\}$, where $v_0:\bb T \to [0,1]$ is continuous, $H:[0,T]\times \bb T \to \bb R$ is of class $\mc C^{1,2}$ and $a:[0,T] \to \bb R$ is of class $\mc C^1$.
Given such an $i$, consider the martingale $\mc M_t^{i,n}$ from \eqref{eq:Mi}. 
Since it is a positive martingale with unit expectation we can use it to define a new probability law $\bb P^i_n$ on $\mc D([0,T]; \Omega_n \times \bb T_n)$, by
\begin{align*}
 \frac{d\bb P^i_n}{d\bb P_n} := \mc M_T^{i,n}.
\end{align*}
We call the \emph{perturbed system}, the time-inhomogeneous Markov process on $\Omega_n \times \bb T_n$ described by $\bb P^i_n$ with generator
\begin{align}\label{perturbed}
 L_{i,n,t} f(\eta; x) &= n^2 \sum_{y\sim z} e^{H_t(\eta^{x,y})-H_t(\eta)} \big(f(\eta^{y,z};x)-f(\eta;x)\big) \\
 &+n\sum_{z=\pm 1} e^{za(t)} c^z(\eta;x)\big(f(\eta;x+\tfrac{z}{n})-f(\eta;x)\big),
\end{align}
where $H_t(\eta) = \int \sum_{x\in \bb T_n} \eta(x)\delta_x^n(y) H_t(y) \,dy$.
We want to derive the hydrodynamic behaviour of this perturbed system, namely, the analogs of Propositions \ref{p2} and \ref{p3}.
For this aim, we first show that the statement of Lemma \ref{l1} remains in force under $\bb P^i_n$.

\begin{lemma}
\label{l1extended}
Let $f: \Omega_n \to \bb R$ be a local function. Then,
\begin{equation*}
\varlimsup_{\epsilon \to 0} \varlimsup_{\vphantom{0}n \to \infty} \frac{1}{n} \log \bb P^i_n\Big(\Big|\int_0^t \big\{f(\xi_s^n) - \bar{f}\big(\hat{\pi}_s^n(\iota_\epsilon)\big)\big\} ds \Big| > \delta \Big) = -\infty
\end{equation*}
for any $\delta>0$ and any $t \in [0,T]$.
\end{lemma}
\begin{proof}[Proof of Lemma \ref{l1extended}]
By \eqref{eq:Mi} and the explicit expressions of the involved factors, we have that

\begin{equation}\label{RadonNikbound}
\big\|\frac{d\bb P^i_n}{d\bb P_n}\big\|_\infty=\|\mc M_T^{i,n}\|_\infty \leq \exp\{n\,C_{\gamma,H,a,T}\}.
\end{equation}

Recall the notations in \eqref{abbrev1} and \eqref{abbrev2} and note that

\begin{equation}\label{tilt}
\frac{1}{n} \log \bb P^i_n\Big(\Big|\int_0^t \mc W_{f}^{\epsilon n}(\xi_s^n) ds \Big| > \delta \Big)=
\frac{1}{n} \log \bb E_n\Big[\frac{d\bb P^i_n}{d\bb P_n} 1_{\big\{|\int_0^t \mc W_{f}^{\epsilon n}(\xi_s^n) ds | > \delta \big\}}\Big].
\end{equation}

The claim now follows by applying \eqref{RadonNikbound} and Lemma \ref{l1} to the r.h.s. of \eqref{tilt}.
\end{proof}
Note that the generator in \eqref{perturbed} restricted to functions acting only on the first coordinate $\eta$ corresponds to a perturbation of the exclusion process.
The hydrodynamic behaviour of such a perturbed exclusion process is well known in the literature, see e.g. \cite{KL}, Proposition 5.1, Chapter 5. We recall it in the next proposition.

\begin{proposition}\label{prop:perturbed-exclusion}
Fix  $i=\{v_0,H\}$, with respect to $\bb P^i_n$, 
\begin{equation*}
\lim_{n \to \infty} \pi_t^n(dx) = u_i(t,x) dx
\end{equation*}
in distribution with respect to the $J_1$-Skorohod topology on  $\mc D([0,T]; \mc M^+(\bb T))$, where the density $\{u_i(t,x); t \in [0,T], x \in \bb T\}$ is the unique 
solution of
\begin{equation}
\begin{cases}\label{HydroPerturbedExclusion}
\partial_t u_i(t,x) &= \Delta u_i(t,x) -\partial_x\left( u_i(t,x)(1-u_i(t,x)\partial_x H\right) \\
 u_i(0,x) &= v_0(x).
\end{cases}
\end{equation} 
\end{proposition}

We are now ready to prove the hydrodynamic behavior for our perturbed system.
\begin{proposition}\label{prop:perturbed-hydrodynamic-limit}
Define $v_a^\pm(\rho,t) := e^{\pm a(t)} v^\pm(\rho)$, $v_a(\rho,t) := v^+_a(\rho,t)-v^-_a(\rho,t)$ and fix an index $i=\{v_0,H,a\}$.
Under $\bb P^i_n$, the triple $(\hat\pi_t^n, \frac{1}{n} N_t^{n,+}, \frac{1}{n} N_t^{n,-})$ converges in distribution to $(\hat u_i(t), \frac{1}{2}(f_i(t)+t),\frac{1}{2}(f_i(t)-t))$ with respect to the $J_1$-Skorohod topology in 
$\mc D([0,T]; \mc M^+(\bb T)  \times \mc M^+(\bb T) \times \mc M^+(\bb T)),$  where $\hat u_i$ is the unique solution of 
\begin{equation}
\begin{cases}\label{perturbedPDE}
 \partial_t \hat u_i(t,x) &= \Delta \hat u_i(t,x) -\partial_x\left( \hat u_i(t,x)(1-\hat u_i(t,x))\partial_x H\right) + v_a(\hat u_i(t,0))\partial_x\hat u_i(t,x)	\\
 \hat u_i(0,x) &= v_0(x),	
\end{cases}
\end{equation}
and $f_i$ is given by
\begin{equation}
\begin{cases}\label{perturbedODE}
 f_i'(t) &= v_a(u_i(t,f_i(t)))=v_a(\hat u_i(t,0))\\
 f_i(0) &= 0.
\end{cases}
\end{equation}
\end{proposition}
\begin{proof}[Sketch of proof of Proposition \ref{prop:perturbed-hydrodynamic-limit}]
Much of the proof is analogous to the proof of the unperturbed system derived in \cite{AveFraJarVol}. For this reason, we only show how to adapt it.
Since $a, H$ are bounded, the tightness arguments for $\hat \pi^n, x^n$ of \cite{AveFraJarVol}, Section 2.3, show tightness in the space $\mc D([0,T]; \mc M^+(\bb T) \times \bb T)$. A more careful checking of these arguments show that we can actually prove tightness in $\mc D([0,T]; \bb T)$ of each of the processes $\{\frac{1}{n} N_t^{n,\pm}; t \in [0,T]\}_{n \in \bb N}$. Since these processes are increasing, the uniform topology is stronger than the weak topology, showing tightness for the triple $(\hat{\pi}_t^n, \frac{1}{n} N_t^{n,+}, \frac{1}{n} N_t^{n,-})$.

To identify the limits, the local replacement lemma for $\xi^n_t$ needs adaptation to the perturbation, however Lemma \ref{l1extended} can be used to provide a suitable analogue: 
\begin{align}\label{eq:perturbed-replacement}
 \overline{\lim_{\epsilon\to0}}\,\overline{\lim_{n\to\infty}}\bb P^i_n \left( \left| \int_0^t \left(e^{\pm 
 a(s)}c^\pm (\xi^n_s;0)-v_a^\pm(\hat\pi^n_s(\iota_\epsilon)))\right)\,ds\right|>\delta\right) = 0.
\end{align}
Next, the martingales
\begin{align}\label{eq:perturbed-martingale}
 \widetilde{M}^{n,a,\pm}_t := \frac{N_t^{n,\pm}}{n} - \frac1n\int_0^t e^{\pm a(s)}c^\pm (\xi^n_s;0)\,ds
\end{align}
have quadratic variations bounded by $\frac{C}{n}$ for a suitable constant $C$ depending on $a$ and the rates $c^{\pm}$. Hence these martingales converge to 0 in probability, with respect to the uniform topology. With $f_i$ a limit point of $\frac{x^n}{n}$, it follows from \eqref{eq:perturbed-replacement} and \eqref{eq:perturbed-martingale} that
\begin{align*}
 f_i(t) = \int_0^t v_a(u_i(s,f_i(s)))\,ds,
\end{align*}
which is the integral version of \eqref{perturbedODE}.
Since \eqref{HydroPerturbedExclusion} admits a unique solution and
$\hat u_i(t,x) = u_i(t,x+f_i(t))$, the claim follows.
\end{proof}

\subsection{Relative entropy and the rate function}
To obtain the lower bound for the large deviation principle in Theorem \ref{t2}, we show in this section that the relative entropy 
of $\bb P^i_n$ with respect to $\bb P_n$ can be interpreted as a rate function.

\begin{lemma}\label{lemma:entropy-to-rate}
Recall definitions \eqref{RateFn1}, \eqref{RateFn2}, \eqref{RateFn4} and \eqref{RateFn3}.
\begin{align}\label{RelEntr}
 \lim_{n\to\infty} \frac1n H\left(\bb P^i_n \middle| \bb P_n \right) &= h(v_0|u_0) + J(H;u_i) + j(a;u_i,f_i)\\
&\leq \mc I_{\mathrm{rw}}(f_i | u_i) + \mc I_{\mathrm{ex}}(u_i).
\end{align}
\end{lemma}
\begin{proof}[Proof of Lemma \ref{lemma:entropy-to-rate}]
Recall \eqref{eq:U^H},\eqref{eq:U^a} and set $\mc U^n_i := \mc U^{H,n}_{\delta,\epsilon}\cap \mc U^{a,n}_{\delta,\epsilon}$. Since 
 \begin{align*}
  \varlimsup_{n \to \infty} \frac1n\log\bb P_n\left((\mc U^n_i)^c\right) = -\infty
 \end{align*}
 and $\frac1n\log\frac{d\bb P^i_n}{d\bb P_n} = \frac1n\log\mc M^{i,n}_T$ is bounded, we have
 \begin{align*}
  \varlimsup_{n \to \infty} \bb P^i_n\left((\mc U^n_i)^c\right) = 0.
 \end{align*}
 Hence,
 \begin{align*}
  \frac1n H\left(\bb P^i_n \middle| \bb P_n \right) &=\frac1n \bb E^i_n \left[\log\frac{d\bb P^i_n}{d\bb P_n}\right]=
 \bb E^i_n \left[ \frac1n\log\frac{d\bb P^i_n}{d\bb P_n}1_{\mc U^i_n}\right] + o_n(1).
 \end{align*}
 On $\mc U^n_i$,
 \begin{align*}
 \frac1n\log\frac{d\bb P^i_n}{d\bb P_n} &= \frac1n\log\frac{d\hat v^n_{v_0}}{dv^n_{u_0}} + a(T)x^n_T - \!\int_0^T \!\! a'(s)x^n_s + \!\sum_{z=\pm1} \!v^z\left(\pi^n_s(\iota_\epsilon(x_s^n))\right)\left(e^{a(s)}-1\right)\,ds\\
 &\quad+\pi^n_T(H_T) - \pi^n_0(H_0)-\int_0^T \pi^n_s(\partial_s H_s) + \pi^n_s(\Delta H_s) \\
&\quad- \frac1n\sum_{x\in\bb T_n}\left(\nabla H_s(x)\right)^2\pi^n_s\left(\iota_\epsilon(x)(1-\iota_\epsilon(x))\right)\,ds 
+ o_{\delta,n}(1)
 \end{align*}
Finally, when taking $n\to\infty$ and $\delta,\epsilon\to0$, we can use Lemma \ref{l1extended} together with the hydrodynamic limit for the perturbed system, Proposition \ref{prop:perturbed-hydrodynamic-limit}, and obtain \eqref{RelEntr}.
\end{proof}

\subsection{The lower bound}
\label{s7.3}
We can finally show the lower bound which together with \eqref{upperBound} concludes the proof of Theorem \ref{t2}.
We will proceed in two steps. We first restrict ourself to paths obtained as solutions of the perturbed system in Proposition \ref{prop:perturbed-hydrodynamic-limit}. Then, in Lemma \ref{approx} below, we show that paths with finite rate function can be approximated by paths which arise via perturbation.
For notational convience, define 
$$\mc I(x,\pi):=\mc I_{\mathrm{rw}}(x |\pi) + \mc I_{\mathrm{ex}}(\pi)$$ 
\begin{lemma}\label{pertSol}
 Let $\mc O$ be an open set in $\mc E$. Then
\begin{align*}
\varliminf_{n\to\infty} \frac1n \log \bb P_n(\mc O) \geq - \inf\left\{\mc I(f_i,u_i) : H \in \mc C^{1,2}, v_0 :\!\bb T\to [0,1], a \in \mc C^1, (u_i,f_i) \in \mc O \right\}.
\end{align*}
\end{lemma}
\begin{proof}[Proof of Lemma \ref{pertSol}]
For a given open set $\mc O \in\mc E$, choose parameters $i=\{v_0,H,a\}$ such that the solution $(u_i,f_i)$ of the differential equations in \eqref{HydroPerturbedExclusion} and 
\eqref{perturbedODE} is contained in $\mc O$.
By a change of measure and Jensen's inequality we have that 
\begin{align*}
 \log \bb P_n(\mc O) &= \log \bb E_n^i \left[1_{\{(\pi^n,x^n)\in \mc O\}} \frac{d\bb P_n}{d\bb P^i_n}\right] \\
 &= \log \bb E_n^i \left[\frac{d\bb P_n}{d\bb P^i_n}\,\middle|\, \mc O \right] \bb P_n^i(\mc O) \\
 &\geq  \bb E_n^i \left[ \log \left(\frac{d\bb P_n}{d\bb P^i_n}\right)\,\middle|\, \mc O \right] +\log \bb P_n^i(\mc O) 
\end{align*}
Moreover, by Proposition \ref{prop:perturbed-hydrodynamic-limit} and our choice of parameters, $\lim_{n\to\infty}\bb P^i_n(\mc O)=1$ and 
\begin{align*}
 \lim_{n\to\infty} \frac1n\left(\bb E_n^i \left[ \log \left(\frac{d\bb P^i_n}{d\bb P_n}\right)\,\middle|\, \mc O \right] - H\left(\bb P^i_n | \bb P_n\right)\right) =0.
\end{align*}
Hence, by Lemma \ref{lemma:entropy-to-rate},
\begin{align*}
 \varliminf_{n\to\infty} \frac1n \log \bb P_n(\mc O) \geq - \mc I(f_i,u_i).
\end{align*}
Optimizing over $i=\{v_0,H,a\}$ such that $(u_i,f_i) \in \mc O$ ends the proof.
\end{proof}

All that remains is to remove the restriction to paths obtained by perturbations. 
\begin{lemma}\label{approx}
Fix a pair $(x,\pi)$ and a sequence $(y_N, \pi^N)$ which converges to $(x,\pi)$in $\mc D([0,T]; \mc M^+(\bb T)  \times \bb T)$ and for which $\mc I(x,\pi), \mc I(y_N,\pi^N)<\infty$. Furthermore assume that $\epsilon \leq \pi^N \leq 1-\epsilon$, $\epsilon \leq \pi \leq 1-\epsilon$ for some $\epsilon>0$ and that
\begin{align}\label{eq:L1-conv}
 \lim_{N\to\infty} \left \lVert \pi^N(y_N)-\pi(x) \right \rVert_{L^1([0,T])} =0.
\end{align}
Then
\begin{align}\label{IrwConv}
\lim_{N\to\infty}\mc I_{\mathrm{rw}}(y_N |\pi^N)=\mc I_{\mathrm{rw}}(x |\pi).
\end{align}
\end{lemma}
\begin{proof}[Proof of Lemma \ref{approx}]
Let us first observe that \begin{equation}\label{L1Convergence}
\lim_{N\to\infty}\lVert v^z(\pi^N(y_N))-v^z(\pi(x)) \rVert_{L^1([0,T])}=0,\quad\text{for}\quad z=\pm,
\end{equation}
which follows from \eqref{eq:L1-conv} by the Lipschitz continuity of $v^z(\rho)$. Also, by the assumption that the densities are bounded away from 0 and 1, we have 
\[\label{eq:ellipticity2} v^z(\pi^N), v^z(\pi)>\tilde\epsilon, \quad\text{for some }\tilde\epsilon>0\text{ and } z=\pm. \]

In Section \ref{RateFn} we found that under this assumption, when the rate function $\mc I_{\mathrm{rw}}(x|\pi)$ is finite, it can be written explicitly as in equation \eqref{eq:explicitrate1}.
We can thus rewrite \eqref{eq:explicitrate1} as
\begin{equation}\label{RateFnDec}\mc I_{\mathrm{rw}}(x|\pi)=\sum_{j=1}^4\int_0^T h^{(j)}_{\pi_t(x_t)}(x'_t)\,dt,\end{equation}
with 
\begin{equation}\label{h12}
h^{(1)}_\rho(x) := \frac{x+\sqrt{x^2+4v^+(\rho)v^-(\rho)}}{2}, \quad h^{(2)}_\rho(x):=2h^{(1)}_\rho(-x),
\end{equation}
\begin{equation}\label{h34}
h^{(3)}_\rho(x) := x\log\left(\frac{h^{(1)}_\rho(x)}{v^+(\rho)}\right), \quad h^{(4)}_\rho(x) := v^-(\rho)+v^+(\rho).
\end{equation}
In view of \eqref{RateFnDec}, to show \eqref{IrwConv}, we will prove that 
\begin{equation*}
\lim_{N\to\infty}\left|\int_0^T \left[h^{(j)}_{\pi^N_t(y_N(t))}\left(y_N'(t)\right)-h^{(j)}_{\pi_t(x_t)}(x'_t)\right]\,dt\right|=0\quad \text{for } j=1,2,3,4.
\end{equation*}

For $j=4$, this is readily obtained due to \eqref{L1Convergence}.
For $j=1,2,3$, by triangular inequality, we have that
\begin{align}\label{TriApprox}
&\left|\int_0^T \left[h^{(j)}_{\pi^N_t(y_N(t))}\left(y_N'(t)\right)-h^{(j)}_{\pi_t(x_t)}(x'_t)\right]\,dt\right|\\& 
\leq\int_0^T \left|h^{(j)}_{\pi_t(x_t)}(y_N'(t))-h^{(j)}_{\pi_t(x_t)}(x'_t)\right|\,dt
+\int_0^T \left|h^{(j)}_{\pi^N_t(y_N(t))}\left(y_N'(t)\right)-h^{(j)}_{\pi_t(x_t)}(y_N'(t))\right|\,dt.
\end{align}
We want to show that, as $N\rightarrow \infty$, the two terms in the r.h.s. of \eqref{TriApprox} vanish.

For the first term, consider the case $j=3$, the derivative 
\begin{equation*}
\partial_x h_\rho^{(3)}(x) = \log\left(\frac{x+\sqrt{x^2+4v^+(\rho)v^-(\rho)}}{2v^+(\rho)}\right) + \frac{2x}{\sqrt{x^2+4v^+(\rho)v^-(\rho)}},
\end{equation*}
is monotone increasing and by \eqref{eq:ellipticity2}, there are constants $a,b>0$ so that
\[ |\partial_x h^{(3)}_{\rho}(x)|\leq \log(1+a|x|)+b \]
uniformly in $\rho$ and $x$.

By Taylor expansion,
\begin{align}\label{Holder}
\left| \int_0^T h^{(3)}_\rho(x'_t)-h^{(3)}_\rho(y'_t)\,dt \right| &= \left| \int_0^T \partial_x h^{(3)}_\rho(\xi_t)(x'_t-y'_t)\,dt \right|	\\
 &\leq  \int_0^T (\log(1+a \max(|x_t'|,|y_t'|))+b)|x_t'-y_t'|  \,dt	\\
 &\leq 2 \lVert x'_t-y'_t\rVert_\Phi \lVert \log(1+a \max(|x_t'|,|y_t'|))+b \rVert_{\Phi^*}.
\end{align}
Here $\Phi(x)=x\log(1+x)$, $\Phi^*$ is its convex conjugate or Legendre-Fenchel transform and $\lVert\cdot\rVert_\phi$ denotes the Orlicz norm associated to the Young function $\Phi$. See e.g. \cite{KR} for an overview of basic results in Orlicz spaces.
In the last inequality we used the H\"older inequality in the Orlicz space $L_{\Phi}:=\{f:[0,T]\rightarrow \bb T :\lVert f\rVert_\Phi<\infty\}$.
By $\max(|x_t'|,|y_t'|) \leq |x'_t|+|x_t'-y_t'|$ and $\Phi^*(x)\leq e^x-1$ one sees that the right term stays bounded as $y'\to x'$ in $\lVert\cdot\rVert_\Phi$. Hence $y'\to x'$ in $\lVert\cdot\rVert_\Phi$ implies $\left| \int_0^T h^{(3)}_\rho(x'_t)-h^{(3)}_\rho(y'_t)\,dt \right|\to0$.

Thus, to conclude that the first term in the r.h.s. of \eqref{TriApprox} goes to zero when $j=3$, it suffices to show that smooth functions are
dense in the Orlicz space.
This is a consequence of the following two facts.
First, on the set of functions uniformly bounded by an arbitrary but fixed constant, the $L^1$-norm and the Orlicz-norm are equivalent, see e.g. \cite{KR}. 
Hence the fact that the smooth functions lie densely in the bounded functions in
$L^1$ implies the same fact for the Orlicz space. Second, bounded functions are
dense in the Orlicz space, see \cite{KR}.

When $j=1,2$, this argument becomes simpler because $\partial_x h^{(j)} _\rho(x)$ is monotone, 
and $|\partial_x h^{(j)}_{\rho}(x)|\leq K$ uniformly in $x$ and $\rho$, for some positive constant $K$.

For the second term in the r.h.s. of \eqref{TriApprox}, when $j=1,2,3$, we argue as follows.
First, consider the case $j=1$, abbreviate $c_i(t):=v^+(\rho_i(t))v^-(\rho_i(t))$ for $i=1,2$, and estimate

\begin{align}\notag&\left| \int_0^T \left[h^{(1)}_{\rho_1(t)}(x'_t)-h^{(1)}_{\rho_2(t)}(x'_t)\right]\,dt \right|=
\left| \int_0^T \frac{4\left[c_1(t)-c_2(t)\right]}{\sqrt{(x'_t)^2+c_1(t)}+\sqrt{(x'_t)^2+c_2(t)}}\,dt\right|\\ \label{h1b}&
\leq K_1 \int_0^T \left|c_1(t)-c_2(t)\right|\,dt
\leq K_1 \int_0^T \left|v^-(\rho_1)\right|\left|v^+(\rho_1)-v^+(\rho_2)\right|\,dt\\
&+K_1 \int_0^T \left|v^+(\rho_2)\right|\left|v^-(\rho_1)-v^-(\rho_2)\right|\,dt\leq 
K_2\sum_{z=\pm}\lVert v^z(\rho_1)-v^z(\rho_2) \rVert_{L^1([0,T])},
\end{align}

for some constants $K_1,K_2>0$ depending on \eqref{eq:ellipticity2} and on the uniform bounded function $1/\left(\sqrt{x^2+c_1}+\sqrt{x^2+c_2}\right)$.
Hence, the claim follows by \eqref{L1Convergence}. The case $j=2$ is the same due to \eqref{h12}.

It remains to consider the case $j=3$. By using H\"older inequality, estimate 
\begin{align}\label{eq:186}
\left| \int_0^T h^{(3)}_{\rho_1(t)}(x'_t)-h^{(3)}_{\rho_2(t)}(x'_t)\,dt \right| &\leq 2 \lVert x' \rVert_\Phi
\left\lVert \log\left(\frac{x' + \sqrt{x'^2+4c_1(t)}}{x' + \sqrt{x'^2+4c_2(t)}}\cdot \frac{v^+(\rho_2)}{v^+(\rho_1)}\right) \right\rVert_{\Phi^*}.
\end{align}
By the triangle inequality, the right hand side goes to 0 if both
\begin{align}\label{eq:k11}
\left\lVert \log\left(\frac{h^{(1)}_{\rho_1}}{h^{(1)}_{\rho_2}}\right) \right\rVert_{\Phi^*}, \quad\quad
\left\lVert \log\left( \frac{v^+(\rho_2)}{v^+(\rho_1)}\right) \right\rVert_{\Phi^*}
\end{align}
converge to 0. Note that by definition of Orlicz norm, $\lVert (\cdot) \rVert_{\Phi^*}$ goes to 0 iff $\int_0^T \Phi^*(a \cdot(\cdot))\,dt$ goes to 0 for all $a>1$. 
For the left term in \eqref{eq:k11} we have
\begin{align*}
 \Phi^*\left(a \log\left(\frac{h^{(1)}_{\rho_1}}{h^{(1)}_{\rho_2}}\right)\right) &\leq \max\left( \left(\frac{h^{(1)}_{\rho_1}}{h^{(1)}_{\rho_2}}\right)^a - 1, \left(\frac{h^{(1)}_{\rho_2}}{h^{(1)}_{\rho_1}}\right)^a-1\right)
 \end{align*}
 Note that by \eqref{eq:ellipticity2}, $0<c^{-1}\leq \frac{h^{(1)}_{\rho_1}}{h^{(1)}_{\rho_2}}\leq c<\infty$. Hence the above is less than
 \begin{align}\label{k1}
 ac^{a-1}\max\left(\frac{h^{(1)}_{\rho_1}}{h^{(1)}_{\rho_2}}-1, \frac{h^{(1)}_{\rho_2}}{h^{(1)}_{\rho_1}}-1\right) . 
 \end{align}
 
 We have for $\{i,i'\}=\{1,2\}$
 \begin{align*}
  \frac{h^{(1)}_{\rho_i}}{h^{(1)}_{\rho_{i'}}}-1 &= \frac{h^{(1)}_{\rho_i}-h^{(1)}_{\rho_{i'}}}{h^{(1)}_{\rho_{i'}}}	
  = (c_i-c_{i'})\frac{-x+\sqrt{x^2+c_{i'}}}{c_{i'}(\sqrt{x^2+c_i}+\sqrt{x^2+c_{i'}})}.
 \end{align*}
As this right hand side is bounded from above by some constant, \eqref{k1} is estimated by
\[ K_3|c_1-c_2|, \]
and from \eqref{h1b} we can conclude that the left norm in \eqref{eq:k11} goes to 0. 
The right norm in \eqref{eq:k11} is controlled with the same type of argument with $v^+$ instead of $h^{(1)}$, which completes the case $j=3$.
\end{proof}

\begin{lemma}\label{lemma:holderpaths}
 Assume $\pi$ satisfies $\mc I_{\mathrm{ex}}(\pi)<\infty$ and $\pi$ is differentiable in time with an absolutely continuous derivative which satisfies $\lVert \dot\pi \rVert_\infty<M$ for some $0<M<\infty$. Then $\pi_t$ is H\"older-$1/2$ continuous for almost every $t$.
\end{lemma}
\begin{proof}[Proof of Lemma \ref{lemma:holderpaths}]
 Assume $\pi$ has finite energy, that is $\lVert \nabla \pi \rVert_{0,T}<\infty$. Then, by the Sobolev embedding theorem, the conclusion follows. 
So what we will show is that if $\pi$ has infinite energy under the given assumptions, then the rate function is infinite as well, which is a contradiction. 

First observe that instead of taking the supremum over all $H$ when determining $\Iex$ we can restrict ourself to those $H$ with $H_t(0)=0, 0\leq t\leq T$. This is easily seen by observing that $\pi$ has constant mass and hence $J(H-H(0);\pi)=J(H;\pi)$.

Looking in more detail at $J(H;\pi)$, by partial integration, the boundedness of $\dot\pi$ and basic estimates,
\begin{align*}
&\left|-\pi_T(H_T)+\pi_0(H_0)-\int_0^T\pi_t(\partial_t H_t)\,dt\right| = \left|\int_0^T\dot\pi_t(H_t)\,dt\right|	\\
&\leq M\int_0^T \int_0^1 |H_t(x)|\,dxdt = M\int_0^T \int_0^1 \left| \int_0^x \nabla H_t(y) dy\right|\,dxdt	\\
&\leq M\int_0^T \int_0^1 \left| \nabla H_t(y) \right|\,dydt \leq M T^{\frac12}\left(\int_0^T \int_0^1 \left( \nabla H_t(y) \right)^2\,dydt\right)^{\frac12}.
\end{align*}

If $\pi$ has infinite energy, then by Proposition \ref{prop:H} there exists a sequence $H^n\in \mc C^{0,2}$ with $\lVert \nabla H^n \rVert_{0,T}=1$ and $\lim_{n\to\infty}\linn \pi, \nabla\nabla H^n\rinn_{0,T}=\infty$. 
Since
\begin{align*}
\left|J(H^n;\pi)\right| \leq (1+MT^{\frac12}) + 2 \left|\int_0^T \pi_t(\Delta H^n)\,dt\right|,
\end{align*}
we have $\lim_{n\to\infty}J(H^n,\pi)=\infty$.
By approximating functions from $\mc C^{0,2}$ by functions from $\mc C^{1,2}$ we can conclude that $\Iex(\pi)=\infty$, which is a contradiction.
\end{proof}

We are finally in shape to conclude the lower bound.
\begin{proposition}\label{openLB}
 Let $\mc O$ be an open set in $\mc E$. Then
 \begin{align*}
  \varliminf_{n\to\infty} \frac1n \log \bb P_n(\mc O) \geq - \inf_{(\pi,x)\in \mc O}\left\{\mc I_{\mathrm{rw}}(x | \pi) + \mc I_{\mathrm{ex}}(\pi) \right\}.
 \end{align*}
\end{proposition}
\begin{proof}[Proof of Lemma \ref{openLB}]
We extend Lemma \ref{pertSol} in two steps, using Lemma \ref{approx}. We will always keep either $\pi$ or $x$ constant because that way it is easier to show the $L^1$ condition of Lemma \ref{approx}.

First we drop the restriction on $a$. To do so, fix $H,v_0$ and let $u_i$ be the solution of \eqref{HydroPerturbedExclusion}. By Lemma \ref{lemma:holderpaths}, for almost every $t$ $u_i(t,\cdot)$ is H\"older-1/2 continuous, especially $u_i(t,\cdot)$ is continuous. 

Fix a path $x$ with $\mc I_{\mathrm{rw}}(x|u_i)<\infty$.
We have shown in Section \ref{RateFn} that $x$ is is absolutely continuous whenever $\mc I_{\mathrm{rw}}(x |\pi)<\infty$.
Since the class $\mc C^2$ is dense in the set of absolutely continuous functions, we can consider 
a sequence of paths $\{y^{(N)}: N\geq 1\}$ in $\mc C^2$ such that $y^{N}$ converges to $x$ pointwise.

For each $N\geq 1$, let $a_N\in \mc C^1$ be the unique function identified by the solution of 
\begin{equation*}
(y^{N}_t)' = v_{a_N}\left(u_i(t,y^{N}_t)\right).
\end{equation*}
Note that this is possible since $\epsilon\leq u_i\leq 1-\epsilon$ for some $\epsilon>0$. Hence $y^N$ is the solution of \eqref{perturbedODE} corresponding to $a^N,H,v_0$. Since $u_i$ is continuous for almost all $t$, $u_i(t,y^N_t)$ converges pointwise to $u_i(t,x_t)$ a.e. Since $u_i$ is bounded, this implies $L^1$-convergence. By Lemma \ref{approx}, $\mc I_{\mathrm{rw}}(y^N|u_i)$ converges as well, and hence
\begin{align*}
 &\inf\left\{\mc I(f_i,u_i) : H \in \mc C^{1,2}, v_0 :\bb T\to [0,1], a \in \mc C^1, (u_i,f_i) \in \mc O \right\} \\
&=\inf\left\{\mc I(x,u_i) : H \in \mc C^{1,2}, v_0 :\bb T\to [0,1], (u_i,x) \in \mc O \right\}.
\end{align*}

To remove the remaining restrictions we follow the steps in \cite{KL}, Lemma 5.5, Chapter 5, where the corresponding statement for the perturbed exclusion was proved. What we will show is that the approximation steps in that lemma not only work for $\mc I_{\mathrm{ex}}$ but for $\mc I_{\mathrm{rw}}$ as well. The general idea is the following scheme. If $\pi$ is smooth in time and space and is bounded away from 0 and 1 we can find $H,v_0$ so that $\pi=u_i$, where $u_i$ is the solution of \eqref{HydroPerturbedExclusion}. In three steps the conditions are then relaxed, and in each step the convergence of the rate function is proved by use of Lemma \ref{approx}. A minor difference to Lemma 5.5 in \cite{KL} is that we exchange the order of space and time convolution, however that has no influence on the convergence of $\Iex$.

Assume $\pi$ is bounded away from 0 and 1, smooth in time with $\mc I_{\mathrm{ex}}(\pi)<\infty$, and let $x$ be a path with $\mc I_{\mathrm{rw}}(x|\pi)<\infty$. Let $\iota_\epsilon:\mb T \to [0,\infty)$ be a smooth function which integrates to one and has support contained in $[-\epsilon,\epsilon]$. Define $\pi^N(x) = \int \pi(x+y)\iota_{1/N}(y)\,dy$. By Lemma \ref{lemma:holderpaths} $\pi_t$ is H\"older-$1/2$ continuous for a.e. $t$ and 
\[ |\pi^N_t(x_t)-\pi_t(x_t)| \leq C_t \int |y|^{\frac12}\iota_{1/N}\,dy, \]
which converges to 0. 
\newline Therefore we can use Lemma \ref{approx} and obtain that $\lim_{N\to\infty}\Irw(x|\pi^N) = \Irw(x,\pi)$ and hence $\lim_{N\to\infty} \mc I(x,\pi^N) = \mc I(x,\pi)$. Since $\pi^N$ is smooth in space and time and is bounded away from 0 and 1, there are $H,v_0$ so that $\pi^N$ is the solution of \eqref{HydroPerturbedExclusion}. Hence
\begin{align*}
&\inf\left\{\mc I(x,u_i) : H \in \mc C^{1,2}, v_0 :\bb T\to [0,1], (u_i,x) \in \mc O \right\} \\
&= \inf\left\{\mc I(x,\pi) : \exists \epsilon>0: \epsilon\leq \pi\leq 1-\epsilon, \text{$\pi$ smooth in time}, (\pi,x) \in \mc O \right\}.
\end{align*}

Now assume $\pi$ is a density bounded away from 0 and 1, and $x$ is a path with $\mc I_{\mathrm{rw}}(x|\pi)<\infty$. Extend $\pi$ from $[0,T]$ to $[0,T+1]$ by the heat equation. Let $\beta_\epsilon:\mb R\to[0,\infty)$ be a smooth function which integrates to 1 and whose support is contained in $[0,\epsilon]$. Define $\pi^N$ via $\pi^N_t = \int_0^{\frac1N} \pi_{t+s} \beta_{1/N}(s)\,ds$. Since $\pi$ is a c\`adl\`ag path $\pi^N_t(x_t)$ converges to $\pi_t(x_t)$. Hence the condition for Lemma \ref{approx} is satisfied and $\mc I_{\mathrm{rw}}(x|\pi^N)$ converges to $\mc I_{\mathrm{rw}}(x|\pi)$. 

As a final step, assume that $\mc I(x,\pi)<\infty$. Let $\tilde\pi^0$,$\tilde\pi^1$ be the constant paths identical to 0 and 1 respectively. Let $\pi^N = (1-\frac2N)\pi + \frac1N\tilde\pi^0+ \frac1N\tilde\pi^1$. We can no longer apply Lemma \ref{approx}. Instead we prove the statement directly via dominated convergence using estimates similar to the first part of the proof of Lemma \ref{lem:finite-rate}. Let $M$ be the supremum of $v^\pm$ and write $v_t^{\pm,N}:=v_t^{\pm}(\pi^N_t(x_t))$. Then, by \eqref{eq:estimate1},
\begin{align}\label{eq:estimate2} 
0\leq \sum_{z=\pm} v^{z,N}_te^{za_{x,\pi^N}(t)}\leq 2|x_t'|+2M,
\end{align}
which is integrable and independent of $N$. Next, just as in Lemma \ref{lem:finite-rate}, we only need to to find an upper bound on the positive part $(x_t'a_{x,\pi^N}(t))^+$,  which we do only for $x'_t>0$. Using $\log^+\leq |\log| = \log^++\log^-$, 
\begin{align*}
a^+_{x,\pi^N}(t)x_t' &= x_t' \log^+\left(\frac{x_t'+\sqrt{(x_t')^2+4v^{+,N}_tv^{-,N}_t}}{2v^{+,N}_t}\right) \\
&\leq x_t'\left|\log\left(\frac{x_t'+\sqrt{(x_t')^2+4v^{+,N}_tv^{-,N}_t}}{2}\right)\right| + x_t'\left|\log\left((v^{+,N}_t)^{-1}\right)\right|	\\
&\leq x_t' \left[ \log^+(x_t'+M)+\log^-(x_t')\right] + x_t'\left[\log^+\left((v^{+,N}_t)^{-1}\right)+\log^-\left(M^{-1}\right)\right].
\end{align*}
The only dependence on $N$ is in $x_t'\log^+\left((v^{+,N}_t)^{-1}\right)$. The function $v^+(\cdot)$ is a non-negative polynomial which can be 0 only at the boundary points 0 and 1. If $v^+$ is positive everywhere there is nothing to prove. Assume $v^+$ has at least one 0. Since as a polynomial $v^+$ is monotone near 0 and 1 we can find $\delta>0$  so that $(v^{+,N}_t)\geq \min\left(\inf_{\rho\in[\delta,1-\delta]}v^+(\rho),v^+(\pi_t(x_t))\right)$. 
By the assumption that $\mc I_{rw}(x|\pi)<\infty$ and Lemma \ref{lem:finite-rate} we use dominated convergence to show that $\lim_{N\to\infty}\mc I_{rw}(x|\pi^N)=\mc I_{rw}(x|\pi)$.
\end{proof}

\end{document}